\documentclass[11pt]{amsart}
\usepackage{amsopn}
\usepackage{amssymb, amscd}

\usepackage{MnSymbol}
\newcommand{\nc}{\newcommand}

\nc{\eg}{\mathfrak{e} } \nc{\fg}{\mathfrak{f} } \nc{\vg}{\mathfrak{v} } \nc{\wg}{\mathfrak{w} }
\nc{\zg}{\mathfrak{z} } \nc{\ngo}{\mathfrak{n} } \nc{\kg}{\mathfrak{k} }
\nc{\mg}{\mathfrak{m} } \nc{\bg}{\mathfrak{b} } \nc{\ggo}{\mathfrak{g} }
\nc{\ggob}{\overline{\mathfrak{g}} } \nc{\sog}{\mathfrak{so} }
\nc{\sug}{\mathfrak{su} } \nc{\spg}{\mathfrak{sp} } \nc{\slg}{\mathfrak{sl} }
\nc{\glg}{\mathfrak{gl} } \nc{\cg}{\mathfrak{c} } \nc{\rg}{\mathfrak{r} }
\nc{\hg}{\mathfrak{h} } \nc{\tg}{\mathfrak{t} } \nc{\ug}{\mathfrak{u} }
\nc{\dg}{\mathfrak{d} } \nc{\ag}{\mathfrak{a} } \nc{\pg}{\mathfrak{p} }
\nc{\sg}{\mathfrak{s} } \nc{\affg}{\mathfrak{aff} }

\nc{\pca}{\mathcal{P}} \nc{\nca}{\mathcal{N}} \nc{\lca}{\mathcal{L}}
\nc{\oca}{\mathcal{O}} \nc{\mca}{\mathcal{M}} \nc{\tca}{\mathcal{T}}
\nc{\aca}{\mathcal{A}} \nc{\cca}{\mathcal{C}} \nc{\gca}{\mathcal{G}}
\nc{\sca}{\mathcal{S}} \nc{\hca}{\mathcal{H}} \nc{\bca}{\mathcal{B}}
\nc{\dca}{\mathcal{D}} \nc{\zca}{\mathcal{Z}}

\nc{\val}{\operatorname{val}}
\nc{\vp}{\varphi} \nc{\ddt}{\tfrac{d}{dt}} \nc{\dds}{\tfrac{{\rm d}}{{\rm d}s}}
\nc{\dpar}{\tfrac{\partial}{\partial t}} \nc{\im}{\mathtt{i}}

\renewcommand{\Im}{{\rm Im}}

\nc{\SO}{\mathrm{SO}} \nc{\Spe}{\mathrm{Sp}} \nc{\Sl}{\mathrm{SL}}
\nc{\SU}{\mathrm{SU}} \nc{\Or}{\mathrm{O}} \nc{\U}{\mathrm{U}} \nc{\Gl}{\mathrm{GL}}
\nc{\Se}{\mathrm{S}} \nc{\Cl}{\mathrm{Cl}} \nc{\Spein}{\mathrm{Spin}}
\nc{\Pin}{\mathrm{Pin}} \nc{\G}{\mathrm{GL}_n(\RR)} \nc{\g}{\mathfrak{gl}_n(\RR)}

\nc{\RR}{{\Bbb R}} \nc{\HH}{{\Bbb H}} \nc{\CC}{{\Bbb C}} \nc{\ZZ}{{\Bbb Z}}
\nc{\FF}{{\Bbb F}} \nc{\NN}{{\Bbb N}} \nc{\QQ}{{\Bbb Q}} \nc{\PP}{{\Bbb P}}

\nc{\vs}{\vspace{.2cm}} \nc{\vsp}{\vspace{1cm}} \nc{\ip}{\langle\cdot,\cdot\rangle}
\nc{\ipp}{(\cdot,\cdot)} \nc{\la}{\langle} \nc{\ra}{\rangle} \nc{\unm}{\tfrac{1}{2}}
\nc{\unc}{\tfrac{1}{4}} \nc{\und}{\tfrac{1}{16}} \nc{\no}{\vs\noindent}
\nc{\lamkn}{\Lambda^2(\RR^{q+n})^*\otimes\RR^{q+n}} \nc{\lamn}{\Lambda^2(\RR^n)^*\otimes\RR^n} \nc{\lamp}{\Lambda^2\pg^*\otimes\pg}
\nc{\lamg}{\Lambda^2\ggo^*\otimes\ggo} \nc{\lamngo}{\Lambda^2\ngo^*\otimes\ngo}
\nc{\tangz}{{\rm T}^{\rm Zar}} \nc{\mum}{/\!\!/} \nc{\kir}{/\!\!/\!\!/}
\nc{\Ri}{\tfrac{4\Ricci_{\mu}}{||\mu||^2}} \nc{\ds}{\displaystyle}
\nc{\ben}{\begin{enumerate}} \nc{\een}{\end{enumerate}} \nc{\f}{\frac}
\nc{\lb}{[\cdot,\cdot]} \nc{\isn}{\tfrac{1}{||v||^2}}
\nc{\gkp}{(\ggo=\kg\oplus\pg,\ip)} \nc{\ukh}{(\ug=\kg\oplus\hg,\ip)}

\nc{\Hess}{\operatorname{Hess}} \nc{\ad}{\operatorname{ad}}
\nc{\Ad}{\operatorname{Ad}} \nc{\rank}{\operatorname{rank}}
\nc{\Irr}{\operatorname{Irr}} \nc{\End}{\operatorname{End}}
\nc{\Aut}{\operatorname{Aut}} \nc{\Inn}{\operatorname{Inn}}
\nc{\Der}{\operatorname{Der}} \nc{\Ker}{\operatorname{Ker}}
\nc{\Iso}{\operatorname{I}} \nc{\Diff}{\operatorname{Diff}}
\nc{\Lie}{\operatorname{Lie}} \nc{\tr}{\operatorname{tr}} \nc{\dif}{\operatorname{d}}
\nc{\sen}{\operatorname{sen}} \nc{\modu}{\operatorname{mod}}
\nc{\Riem}{\operatorname{Rm}} \nc{\Ricci}{\operatorname{Ric}}
\nc{\sym}{\operatorname{sym}} \nc{\symac}{\operatorname{sym^{ac}}}
\nc{\symc}{\operatorname{sym^{c}}} \nc{\scalar}{\operatorname{R}}
\nc{\grad}{\operatorname{grad}} \nc{\ricci}{\operatorname{Rc}}
\nc{\nr}{\operatorname{nr}} \nc{\riccic}{\operatorname{ric^{c}}}
\nc{\riccig}{\operatorname{ric^{\gamma}}} \nc{\Rin}{\operatorname{M}}
\nc{\Le}{\operatorname{L}} \nc{\tang}{\operatorname{T}}
\nc{\level}{\operatorname{level}} \nc{\rad}{\operatorname{r}}
\nc{\abel}{\operatorname{ab}} \nc{\CH}{\operatorname{CH}}
\nc{\mcc}{\operatorname{mcc}} \nc{\Adj}{\operatorname{Adj}}
\nc{\Order}{\operatorname{O}} \nc{\mm}{\operatorname{M}}
\nc{\inj}{\operatorname{inj}} \nc{\Pf}{\operatorname{Pf}}
\nc{\vol}{\operatorname{vol}} \nc{\Diag}{\operatorname{Diag}}

\theoremstyle{plain}
\newtheorem{theorem}{Theorem}[section]
\newtheorem{proposition}[theorem]{Proposition}
\newtheorem{corollary}[theorem]{Corollary}
\newtheorem{lemma}[theorem]{Lemma}

\theoremstyle{definition}
\newtheorem{definition}[theorem]{Definition}

\theoremstyle{remark}
\newtheorem{remark}[theorem]{Remark}

\newtheorem{example}[theorem]{Example}

\title{On nonsingular two-step nilpotent Lie algebras}

\author{Jorge Lauret} \author{David Oscari}

\address{FaMAF and CIEM, Universidad Nacional de C\'ordoba, 5000 C\'ordoba, Argentina}
\email{lauret@famaf.unc.edu.ar, oscari@famaf.unc.edu.ar}

\thanks{This research was partially supported by grants from CONICET, FonCyT (Argentina)
and SeCyT (Universidad Nacional de C\'ordoba)}

\begin{document}

\maketitle

\begin{abstract}
A $2$-step nilpotent Lie algebra $\ngo$ is called {\it nonsingular} if $\ad{X}:\ngo\rightarrow [\ngo,\ngo]$ is onto for any $X\notin[\ngo,\ngo]$.  We explore nonsingular algebras in several directions, including the classification problem (isomorphism invariants), the existence of canonical inner products (nilsolitons) and their automorphism groups (maximality properties).  Our main tools are the moment map for certain real reductive representations, and the Pfaffian form of a $2$-step algebra, which is a positive homogeneous polynomial in the nonsingular case.
\end{abstract}

\tableofcontents

%7296324951356096616

\section{Introduction}

We study, for a given pair $\ngo_1$, $\ngo_2$ of real vector spaces, bilinear skew-symmetric maps
$$
\mu:\ngo_1\times\ngo_1\longrightarrow\ngo_2,
$$
such that the function $(X,Y)\mapsto \alpha(\mu(X,Y))$ is a non-degenerate $2$-form on $\ngo_1$ for any nonzero $\alpha\in\ngo_2^*$.  This is precisely the condition the O'Neill tensor of a Riemannian submersion must necessarily satisfy at each point in order for all vertizontal $2$-planes to have positive sectional curvature, in which case the bundle is called {\it fat} (see \cite{FlrZll}).  On the other hand, each of such $\mu$'s defines a $2$-step nilpotent Lie algebra $(\ngo=\ngo_1\oplus\ngo_2,\mu)$, often called {\it nonsingular} (also {\it regular}, or {\it fat}) in the literature under the above condition.  The class of nonsingular nilpotent Lie groups strictly contains the well-known {\it H-type} groups introduced by A. Kaplan \cite{Kpl}, and their left-invariant metrics enjoy very nice properties concerning curvature and totally geodesic submanifolds which were proved by P. Eberlein in \cite{Ebr}.

If $\dim{\ngo_1}=m$ and $\dim{\ngo_2}=n$, then the {\it isomorphism} relation between two nonsingular algebras is given by the natural $\Gl_m\times\Gl_n$-action on the vector space $V_{n,m}:=\Lambda^2\ngo_1^*\otimes\ngo_2$ and coincides with Lie algebra isomorphism.  Only very special values for the {\it type} $(n,m)$ of a nonsingular algebra are allowed, which coincide with the types of $H$-type algebras and include $n=1$, $m=2k$; $n=2,3$, $m=4k$; $n=4,\dots, 7$, $m=8k$; etc.  The following natural questions arise from both the algebraic and geometric points of view:

\begin{itemize}
\item[(i)] How `wild' is the classification problem for nonsingular algebras up to isomorphism?

\item[(ii)] Given a nonsingular algebra $\mu:\ngo_1\times\ngo_1\longrightarrow\ngo_2$, is there a canonical inner product on $\ngo_1\oplus\ngo_2$ attached to $\mu$?

\item[(iii)] Is any $H$-type algebra the `most symmetric' one among all nonsingular algebras of the same type?  This in the sense that its automorphism group has maximal dimension, or equivalently, its $\Gl_m\times\Gl_n$-orbit in $V_{n,m}$ is of minimal dimension.
\end{itemize}

The classification of $2$-step algebras is `hopeless' from many rigorous points of view (see e.g. \cite{BltLpySrg,Ebr5}).  A complete classification is only known for the types $(1,m)$ (all Heisenberg algebras up to abelian factors), $(2,m)$ (see \cite{Ggr,LvsTrb} or Section \ref{2m}), $(5,5)$ and $(n,m)$ with $n+m\leq 9$ (over $\CC$, see \cite{GltTms}).  By using pencil invariants, we show that nonsingular algebras of type $(2,m)$ are parameterized by sets of the form $\sca=\{ (\alpha_1,k_1), \dots  ,(\alpha_r,k_r)\}$, where $\alpha_i\in \CC\smallsetminus\RR$, $k_i\in\NN$, and isomorphism classes correspond to $\Gl_2(\RR)$-orbits of such sets relative to the $\Gl_2(\RR)$-action on $\CC$ by M\"obius transformations (see Section \ref{2m}).

For a given inner product $\ip$ on $\ngo=\ngo_1\oplus\ngo_2$ (with $\ngo_1\perp\ngo_2$), one can encode the structural constants of $\mu$ in a map $J_\mu:\ngo_2\longrightarrow\sog(\ngo_1)$ defined by
$$
\la J_\mu(Z)X,Y\ra = \la\mu(X,Y),Z\ra, \qquad \forall X,Y\in\ngo_1, \; Z\in\ngo_2.
$$
There is a nice and useful isomorphism invariant for $2$-step algebras (with $m$ even) called the {\it Pfaffian form}, which is the projective equivalence class of the homogeneous polynomial $f_\mu$ of degree $m/2$ in $n$ variables defined by
$$
f_\mu(Z)^2=\det{J_\mu(Z)}, \qquad \forall Z\in\ngo_2,
$$
for each $\mu$ of type $(n,m)$ (see Section \ref{pf} for a more precise definition).  It is an immediate, though quite an intriguing fact, that $\mu$ is nonsingular if and only if $f_\mu$ is a {\it positive polynomial} (i.e. $f_\mu(x)>0$ for any nonzero $x\in\RR^n$).  Even though the theory of positive polynomials has a very rich presence in the literature, a complete classification of them, up to projective equivalence, seems to be as difficult as the classification of all forms and thus is also unknown and mostly intractable, even in low dimensional cases like ternary quartics (i.e. $(n,m)=(3,8)$).  Nevertheless, in this paper, the Pfaffian form has proved to be a powerful tool to exhibit continuous families of pairwise non-isomorphic nonsingular algebras.  It is proved in \cite{Osc2} that any positive ternary quartic form is the Pfaffian form of at least one nonsingular algebra of type $(3,8)$, showing that a reasonable classification of nonsingular algebras of type $(3,8)$ is hopeless.

Concerning question (ii) above, the meaning of the word `canonical' is part of the problem.  The maps $J_\mu(Z)$ provide a nice tool to consider compatibility conditions between $\mu$ and $\ip$.  For instance, when $J_\mu(Z)^2=-\| Z\|^2I$ for any $Z\in\ngo_2$, the metric Lie algebra $(\ngo,\mu,\ip)$ is called {\it H-type} (see \cite{Kpl,BrnTrcVnh}), and if more generally, $Z\mapsto(-\det{J_\mu(Z)^2})^{1/m}$ is a quadratic form, then the algebra is said to be of {\it $\tilde{H}$-type} (see \cite{LvsTrb}).

We may also consider the following condition involving only $\mu$ and $\ip$, as another generalization of $H$-type: for any (or some) orthonormal basis $\{ Z_i\}$ of $\ngo_2$,
\begin{equation}\label{nilsoint}
\sum J_\mu(Z_i)^2=aI, \qquad \tr{J_\mu(Z_i)J_\mu(Z_j)}=b\delta_{ij}, \qquad \mbox{for some} \quad a,b<0.
\end{equation}
It follows from well-known results in geometric invariant theory (see Section \ref{git}) that if condition (\ref{nilsoint}) holds, then:

\begin{itemize}
\item $\mu$ is a minimal vector for the $\Sl_m\times\Sl_n$-action on $V_{n,m}$ (i.e. $\|\mu\|\leq\|h\cdot\mu\|$ for any $h\in\Sl_m\times\Sl_n$).

\item The $\Sl_m\times\Sl_n$-orbit of $\mu$ is closed in $V_{n,m}$.

\item The set of minimal vectors in $(\Sl_m\times\Sl_n)\cdot\mu$ consists of a single $\SO(m)\times\SO(n)$-orbit.  We will strongly use this uniqueness property to distinguish, up to isomorphism, those special nonsingular algebras satisfying (\ref{nilsoint}) by using $\SO(m)\times\SO(n)$-invariants, which are much more abundant.
\end{itemize}

From a different geometric point of view, condition (\ref{nilsoint}) implies that the left invariant metric $g$ determined by $(\ngo,\mu,\ip)$ on the corresponding simply connected nilpotent Lie group $N$ maximizes the scalar curvature among all left-invariant metrics on $N$ with $\ngo_1\perp\ngo_2$ and the same volume.  It is also a minimum of the square norm of the Ricci tensor among all left-invariant metrics on $N$ with the same scalar curvature.  On the other hand, $(N,g)$ is known to be a {\it Ricci soliton}, i.e. the Ricci flow solution starting at $g$ evolves only by time-dependent scaling and pull-back by diffeomorphisms.  Such metrics are called {\it nilsolitons} in the literature and have been extensively studied in the last decade (see e.g. the survey \cite{cruzchica}).  A metric satisfying (\ref{nilsoint}) has also been called a {\it metric with optimal Ricci tensor}, and in particular, has a geodesic-flow invariant Ricci tensor (see \cite[Section 7]{Ebr4} and \cite[Section 7]{Ebr2}).  An inner product for which (\ref{nilsoint}) holds is unique up to scaling and automorphisms of $\mu$ (i.e. up to homothety), which also supports the presentation of these inner products as canonical or distinguished for a given $\mu$.

We prove in this paper that condition (\ref{nilsoint}) is actually also necessary for any nilsoliton on a nonsingular algebra (see Theorem \ref{nilthm} for a more complete statement).

\begin{theorem}
If $\mu$ is a nonsingular algebra such that $\ip$ is a nilsoliton, then condition {\rm (\ref{nilsoint})} holds.
\end{theorem}

In \cite{Nkl3}, Y. Nikolayevsky gave a complete classification of $2$-step nilpotent Lie algebras of type $(2,m)$ admitting a nilsoliton.  We give in Section \ref{2m} an alternative proof in the nonsingular case.

\begin{theorem}
A nonsingular algebra $\mu_{\sca}$ of type $(2,m)$, where $\sca=\{ (\alpha_1,k_1), \dots  ,(\alpha_r,k_r)\}$, admits a nilsoliton if and only if $k_1=\dots =k_r=1$.
\end{theorem}

We use this result to obtain $k$-parameter families of nonsingular algebras of type $(2,m)$ which do not admit nilsoliton metrics, for any $k\geq 1$.  For the type $(2,8)$, $r\geq 2$, the curve  $\mu_t=\mu_{\sca_t}$, where $\sca_t:=\{(\im,1),(\sqrt{t}\im,1)\}$, $t\geq 1$, is pairwise non-isomorphic and consists of nonsingular algebras which do admit a nilsoliton.  Their Pfaffian forms equal $f_{\mu_t}(x,y)=(x^2+y^2)(x^2+ty^2)$.  On the other hand, notice that $\mu_\sca$ with $\sca:=\{(\im,2)\}$ has identical Pfaffian than $\mu_1$, given by $(x^2+y^2)^2$, but $\mu_\sca$ does not admit a nilsoliton.  We also use Nikolayevsky's classification to obtain curves of $2$-step algebras of type $(2,m)$ (not necessarily nonsingular) which do not admit nilsoliton metrics.  Our examples cover any dimension $\geq 14$ and dimension $12$ (compare with \cite{Jbl,Pyn,Osc1}).

In Section \ref{38}, we give many curves of nonsingular algebras of type $(3,8)$ admitting a nilsoliton metric.  We have also found the first examples of nonsingular algebras of type $(3,8)$ which do not admit nilsolitons.  They consist of two curves having identical Pfaffian, but with different dimensions for their respective automorphism groups.

Let us now consider question (iii) above.  It has recently been proved by A. Kaplan and A. Tiraboschi in \cite{KplTrb} that for any nonsingular algebra $\mu$, the projection of the automorphism group $\Aut(\mu)$ on $\glg_n=\glg(\ngo_2)$ satisfies that
\begin{equation}\label{maxint}
\dim{\Aut(\mu)|_{\ngo_2}}\leq 1+\frac{n(n-1)}{2},
\end{equation}
by showing the existence of an $\Aut(\mu)$-invariant (up to scaling) inner product on $\ngo_2$ (which generalizes L. Saal's result in the case of $H$-type algebras, see \cite{Sal}).  It is well known that equality holds for $H$-type algebras (see \cite{Rhm}), and it was believed for a long time that this might be a way to characterize $H$-type algebras among nonsingular ones.  In Section \ref{aut}, we show that this is not true by proving that the nonsingular algebra $\mu_{\sca}$ of type $(2,4k)$, where $\sca=\{(\im,k)\}$, which is of $H$-type if and only if $k=1$, satisfies the equality in condition (\ref{maxint}) for any $k\in\NN$ (we have recently become aware that this result was independently obtained in \cite{KplTrb}).

On the other hand, it is natural to consider the stronger condition
\begin{equation}\label{maxort}
\dim{K(\mu)|_{\ngo_2}}=\frac{n(n-1)}{2}, \qquad\mbox{or equivalently}, \qquad \SO(n)\subset K(\mu)|_{\ngo_2},
\end{equation}
where $K(\mu)$ denotes the subgroup of $\Aut(\mu)$ of orthogonal automorphisms.  Clearly, this also holds for $H$-type algebras, and it is claimed in \cite[Theorem 2.5]{KplTrb} that $H$-type algebras are the only nonsingular algebras for which (\ref{maxort}) holds.  However, we give many examples of nonsingular algebras in each type $(3,4k)$, $k\geq 2$, satisfying (\ref{maxort}), none of which is $H$-type.  The simplest one is of type $(3,8)$ and is obtained by taking $J_\mu(\ngo_2)\subset\sog(8)$ to be any Lie subalgebra of $\sog(8)$ isomorphic to $\sug(2)$ and acting irreducibly on $\RR^8$.

\vs \noindent {\it Acknowledgements.}  This research is part of the Ph.D. thesis (Universidad Nacional de C\'ordoba) by the second author, under the supervision of the first one (see \cite{Oscth}).  We are very grateful to F. Cukierman,  M. Jablonski  and W. Ziller for fruitful discussions on the topic of this paper, and to the referee for several helpful comments and suggestions.

\section{Invariants of nonsingular algebras}\label{inv}

We consider a real vector space $\ngo$ and fix a direct sum decomposition
$$
\ngo=\ngo_1\oplus\ngo_2, \qquad \dim{\ngo_1}=m, \qquad\dim{\ngo_2}=n.
$$
Every
$2$-step nilpotent Lie algebra of dimension $m+n$ with derived algebra of dimension $\leq n$ can be represented by a bilinear skew-symmetric map
$$
\mu:\ngo_1\times\ngo_1\longrightarrow\ngo_2.
$$
The set of all such maps is the vector space $V_{n,m}:=\Lambda^2\ngo_1^*\otimes\ngo_2$ of dimension $n\binom{m}{2}$, and an element $\mu\in V_{n,m}$ is said to be of {\it type} $(n,m)$ if $\mu(\ngo_1,\ngo_1)=\ngo_2$.  We note that these elements form an open and dense subset of $V_{n,m}$, and two of them are isomorphic as Lie algebras if and only if they lie in the same $\Gl_m\times\Gl_n$-orbit with respect to the natural action (see \eqref{actiong} for the corresponding action of the Lie algebra):
\begin{equation}\label{actionG}
(\psi,\vp)\cdot\mu := \vp\mu(\psi^{-1}\cdot,\psi^{-1}\cdot), \qquad (\psi,\vp)\in\Gl_m\times\Gl_n, \quad\mu\in V_{n,m}.
\end{equation}
Here and throughout the rest of the paper, $\Gl_n$ will denote the group of invertible linear operators of an $n$-dimensional vector space, and also the real general linear group $\Gl_n(\RR)$ if there is a fixed basis in the context.

If we fix bases $\{ X_1,\dots ,X_m\}$ and $\{ Z_1,\dots ,Z_n\}$ of $\ngo_1$ and $\ngo_2$, respectively, then each $\mu\in V_{n,m}$ is determined by its structural constants $\mu_{ij}^k\in\RR$ defined by
$$
\mu(X_i,X_j)=\sum_k\mu_{ij}^k Z_k.
$$
An alternative way to arrange the structural constants of $\mu$ is by fixing the inner product $\ip$ on $\ngo$ that makes the above basis orthonormal in order to define $J_\mu:\ngo_2\longrightarrow\sog(\ngo_1)$ by
$$
\la J_\mu(Z)X,Y\ra = \la\mu(X,Y),Z\ra, \qquad \forall X,Y\in\ngo_1, \; Z\in\ngo_2.
$$
Indeed, the $ij$-entry of the matrix of $J_\mu(Z_k)$ is precisely given by $-\mu_{ij}^k$.

\begin{example}\label{444}
The $11$-dimensional $2$-step algebra $\mu\in V_{3,8}$ defined by
$$
\begin{array}{lll}
\mu(X_1,X_3)=-Z_1, & \mu(X_1,X_5)=-Z_2, & \mu(X_1,X_7)=-2Z_2+\sqrt{2}Z_3, \\
\mu(X_2,X_4)=-Z_1, & \mu(X_2,X_6)=-Z_2, & \mu(X_2,X_8)=-Z_2+\frac{1}{\sqrt{2}}Z_3, \\
\mu(X_3,X_6)=-Z_3, & \mu(X_3,X_8)=-\frac{1}{\sqrt{2}}Z_2, & \mu(X_4,X_6)=-Z_3, \\
\mu(X_4,X_8)=-\frac{1}{\sqrt{2}}Z_2, & \mu(X_5,X_8)=-Z_1, & \mu(X_6,X_7)=-Z_1,
\end{array}
$$
can be represented by its structural map $J_\mu$ given by
$$
J_\mu(xZ_1+yZ_2+zZ_3)=\left[ \begin {smallmatrix}
0&0&x&0&y&0&2y-\sqrt {2}z&0\\
0&0&0&x&0&y&0&y+\frac{1}{\sqrt{2}}z\\
-x&0&0&0&z&0&\sqrt{2}y&0\\
0&-x&0&0&0&z&0&-\frac{1}{\sqrt{2}}y\\
-y&0&-z&0&0&0&0&x\\
0&-y&0&-z&0&0&x&0\\
-2y+\sqrt{2}z&0&-\sqrt{2}y&0&0&-x&0&0\\
0&-y-\frac{1}{\sqrt{2}}z&0&\frac{1}{\sqrt{2}}y&-x&0&0&0
\end{smallmatrix} \right].
$$
\end{example}

We focus in this paper on the following special class of algebras.

\begin{definition}\label{ns}
A $2$-step algebra $(\ngo=\ngo_1\oplus\ngo_2,\mu)$, $\mu\in V_{n,m}$, is called {\it nonsingular} if any of the following conditions holds:
\begin{itemize}
\item The map $\ngo_1\longrightarrow\ngo_2$, $Y\mapsto\mu(X,Y)$ is onto for any nonzero $X\in\ngo_1$.

\item For each nonzero $Z\in\ngo_2$, the map $(X,Y)\mapsto\la\mu(X,Y),Z\ra$ is a non-degenerate $2$-form on $\ngo_1$.

\item $J_\mu(Z)$ is invertible for any nonzero $Z\in\ngo_2$.
\end{itemize}
\end{definition}

It is easy to see that these conditions are indeed equivalent and do not depend on the chosen inner product.  Note that $m$ must be even for nonsingularity to hold.  It follows from the strong restrictions for the maximal number of linearly independent vector fields on spheres that nonsingular algebras can only exist in very special dimensions.  Indeed, if $m=(2a+1)2^{4b+c}$, where $0\leq c\leq 3$, then $n\leq 2^c+8b-1$ (see e.g. \cite{LwsMch}).  In particular, for $n\leq 11$ one has
$$
\begin{array}{c|c|c|c|c|c|c}
n & 1 & 2-3 & 4-7 & 8 & 9 & 10-11 \\ \hline
m & 2k & 4k & 8k & 16k & 32k & 64k
\end{array}, \qquad k\in\NN.
$$
The subset $V^+_{n,m}\subset V_{n,m}$ of nonsingular algebras is clearly open in $V_{n,m}$, although it is not always dense.  For instance, $V^+_{2,4}$ consists of a single $\Gl_4\times\Gl_2$-orbit of dimension $12$, the one corresponding to the (real) H-type algebra $\hg_3\otimes\CC$ (where $\hg_3$ denotes the $3$-dimensional real Heisenberg algebra), but the orbit of $\hg_3\oplus\hg_3$ is also $12$-dimensional (and open) and so $V^+_{2,4}$ is not dense in $V_{2,4}$.

There is a great lack of known invariants to distinguish isomorphism classes in $V_{n,m}$, even for nonsingular elements.  A complete classification is only known for the types $(1,m)$ (all Heisenberg algebras up to abelian factors), $(2,m)$ (see \cite{Ggr} or Section \ref{2m}), $(5,5)$ and $(n,m)$ with $n+m\leq 9$ (see \cite{GltTms}).  By an invariant we mean a function from $V_{n,m}$ to some set $X$ which is constant on $\Gl_m\times\Gl_n$-orbits and therefore provides a necessary condition on two algebras to be isomorphic; namely, the values of the invariant at two isomorphic algebras must coincide.

We give in the next subsections a number of different approaches to get invariants of $2$-step algebras.

\subsection{Pfaffian form}\label{pf}
We associate to each $\mu\in V_{n,m}$ its {\it Pfaffian form} $f_{\mu}$ defined by
$$
f_{\mu}(x_1,\dots ,x_n)=\Pf\left(J_\mu({x_1Z_1+\dots +x_nZ_n})\right),
$$
where $\Pf:\sog(\ngo_1)\longrightarrow \RR$ is the usual Pfaffian, that is, the only polynomial function satisfying $\Pf(B)^2=\det{B}$ for all $B\in\sog(\ngo_1)$ and $\Pf(J)=1$ for some fixed $J\in\sog(\ngo_1)$ having only $\pm\im$ as eigenvalues (see e.g. \cite[10.3]{Muk} for more information on Pfaffians of skew-symmetric matrices).

\begin{remark}
An alternative way to define $f_\mu$ is as a polynomial function on $\ngo_2$: $f_\mu(Z)=\Pf\left(J_\mu(Z)\right)$.  If we change the basis in the definition we gave above by $\{ gZ_1,\dots,gZ_n\}$, $g\in\Gl_n$, then we obtain the polynomial $f_\mu(g^{-1}(x_1,\dots,x_n))$.
\end{remark}

Recall that $J_\mu(Z)$ is an $m\times m$ matrix, and since
$$
f_\mu(Z)^2=\det{J_\mu(Z)},
$$
we need $m$ to be even in order
to get $f_\mu\ne 0$.

We therefore assume that $m$ is even, say $m=2d$, and thus the Pfaffian determines a continuous function
\begin{equation}\label{pff}
f:V_{n,2d}\longrightarrow P_{n,d}, \qquad \mu\mapsto f_\mu,
\end{equation}
where $P_{n,d}:=P_{n,d}(\RR)$ is the algebra of all homogeneous polynomials of degree $d$ in $n$ variables with coefficients in $\RR$ (sometimes called
$n$-{\it ary} $d$-{\it ic forms} for short).  There is a natural left $\Gl_n$-action on $P_{n,d}$ given by $\vp\cdot f:= f\circ\vp^{-1}$.

\begin{example}\label{444-pf}
By a straightforward computation one gets that the Pfaffian form of the algebra in Example \ref{444} is given by
$$
f_\mu(x,y,z)=x^4+y^4+z^4,
$$
from which it follows that $\mu$ is nonsingular as $f_\mu$ only vanishes at zero (see the third condition in Definition \ref{ns}).
\end{example}

It is proved in \cite{Sch} that the projective equivalence class of the form
$f_\mu(x_1,\dots ,x_n)$ is an isomorphism invariant of the Lie algebra $(\ngo,\mu)$ (see also \cite[Proposition 2.4]{ratform}).  More precisely, if $\mu,\lambda\in V_{n,m}$ are isomorphic, then $f_\mu$ and $f_\lambda$ are {\it projectively equivalent} (denoted by $f_\mu\simeq f_\lambda$), i.e.
there exist $\vp\in\Gl_n$ and $c\ne 0$ such that
$$
f_\lambda(x_1,\dots ,x_n)=cf_\mu(\vp(x_1,\dots ,x_k)), \qquad \forall (x_1,\dots ,x_n)\in\RR^n.
$$
Indeed, it is easy to see that if $\lambda=(\psi,\vp)\cdot\mu$ for some $(\psi,\vp)\in\Gl_{2d}\times\Gl_n$, then $J_\lambda(Z)=(\psi^{-1})^tJ_\mu(\vp^tZ)\psi^{-1}$ for all $Z$ and thus $f_\lambda=(\det{\psi})^{-1}\, f_\mu\circ\vp^t$.

\begin{example}\label{su2}
The Pfaffian form of the algebra $\mu\in V_{3,8}$ defined by
$$
J_\mu(xZ_1+yZ_2+zZ_3)=\left[
    \begin{smallmatrix}
      0          &     3x       & -\sqrt{3}y  & -\sqrt{3}z &             &             &  &  \\
      -3x        &     0       & \sqrt{3}z    & -\sqrt{3}y &             &             &  &  \\
       \sqrt{3}y & -\sqrt{3}z  & 0           & x           & -2y         & -2z         &  &  \\
       \sqrt{3}z &   \sqrt{3}y & -x          & 0           & 2z          & -2y         &  &  \\
                 &             &  2y         &  -2z        & 0           & -x          & -\sqrt{3}y & -\sqrt{3}z \\
                 &              & 2z         &  2y         & x           & 0           & \sqrt{3}z & -\sqrt{3}y \\
                 &              &            &             & \sqrt{3}y   & -\sqrt{3}z  & 0 & -3x \\
                 &              &            &             &  \sqrt{3}z  & \sqrt{3}y   & 3x & 0 \\
    \end{smallmatrix}
  \right]
$$
is given by $f_\mu(x,y,z)=9(x^2+y^2+z^2)^2$, and consequently, $\mu$ is also nonsingular and non-isomorphic to the algebra from Examples \ref{444}-\ref{444-pf}.  There are two easy ways to see that these two ternary quartics are not projectively equivalent: by computing their isotropy subgroups in $\Gl_3$, and by looking at their level sets $f_\mu=c$ in $\RR^3$.
\end{example}

One can therefore use invariants of forms to distinguish algebras in $V_{n,m}$ up to isomorphism.  Consider the ring $\RR\left[P_{n,d}\right]^{\Sl_n}$ of $\Sl_n$-invariant polynomials on $P_{n,d}$, that is, the set of all polynomial functions $I:P_{n,d}\longrightarrow\RR$ such that
$$
I(\vp\cdot f)=I(f) \qquad\forall\vp\in\Sl_n,\quad f\in P_{n,d}.
$$
It is well known that $I$ is invariant if and only if each of its homogeneous components are so.  However, a set of generators and their relations for
such a ring is only known for small values of $n$ and $d$, including $d=2$ and any $n$, $n=2$
and $d\leq 8$, $n=3$ and $d\leq 3$.  We refer to \cite{Dlg,Muk} for many explicit classification results on forms and invariants of forms.

\begin{lemma}\label{proyinv}
For $f,g\in P_{n,d}$, $n$ odd, assume that there exist homogeneous invariants $I,I'\in\RR\left[P_{n,d}\right]^{\Sl_n}$ of the same degree such that $I'(f),I'(g)\ne 0$ and
$$
\frac{I(f)}{I'(f)}\not= \frac{I(g)}{I'(g)}.
$$
Then $f$ is not projectively equivalent to $g$.
\end{lemma}

The proof follows easily by using the fact that $I(c\vp\cdot f)=c^kI(f)$ for any $c\in\RR$, $\vp\in\Sl_n$, $f\in P_{n,d}$, where $k$ is the degree of the invariant $I$, and the fact that for $n$ odd one has $\Gl_n=\RR^*\Sl_n$.

\begin{example}\label{exinvform}
Suppose we have a one-parameter family $\mu_t\in V_{3,8}$ of $2$-step algebras with Pfaffian forms
$$
f_{\mu_t}(x,y,z)=(x^2+y^2+z^2)^2+tx^2y^2 \in P_{3,4}.
$$
It follows that $\mu_t$ is nonsingular for any $t\geq 0$ (recall that $\det{J_{\mu_t}(Z)}=f_{\mu_t}(Z)^2$), but the question is: is this really a `curve' of algebras, in the sense that $\mu_t$ is not isomorphic to $\mu_s$ for all $t\ne s$?  If we write any ternary quartic $f\in P_{3,4}$ as
\begin{align}
f(x,y,z) = & ax^4+4bx^3y+6cx^2y^2+4dxy^3+ey^4 +4fx^3z+12gx^2yz+12hxy^2z \label{par-34}\\
& +4iy^3z+6jx^2z^2+12kxyz^2+6ly^2z^2+4mxz^3+4nyz^3+pz^4, \notag
\end{align}
then from \cite{Dxm} we get the following $\Sl_3$-invariant homogeneous polynomials on $P_{3,4}$ of degree $3$ and $6$, respectively:
\begin{align*}
I_3(f) := & aep+3(al^2+ej^2+pc^2)+4(bim+fdn) \\
& -4(ain+efm+pbd)+6cjl+12(ck^2+jh^2+lg^2)-12ghk  \\
& -12(bkl+fhl+dkj+igj+mhc+ngc)+12(gdm+hnb+kfi),
\end{align*}
and
\begin{equation}\label{calecticant}
I_6(f):=\det{H(f)}, \qquad\mbox{where} \quad H(f):=\left|
                               \begin{array}{cccccc}
                                 a & c & j & g & f & b \\
                                 c & e & l & i & h & d \\
                                 j & l & p & n & m & k \\
                                 g & i & n & l & k & h \\
                                 f & h & m & k & j & g \\
                                 b & d & k & h & g & c \\
                               \end{array}
                             \right|.
\end{equation}
The invariant $I_6(f)$ is called the {\it calecticant} and $H(f)$ the {\it Hankel} matrix (or quadratic form) of $f$.
There are other five invariants described in \cite{Dxm} of degree $9$, $12$, $15$, $18$ and $27$, respectively, which are more difficult to handle.  Anyway, by a straightforward computation, we get that
$$
I_6\left( f_{\mu_t}\right)=-\tfrac{1}{1944}t^3-\tfrac{7}{2916}t^2+\tfrac{8}{729}t+\tfrac{20}{729}, \qquad
I_3\left( f_{\mu_t}\right)=\tfrac{1}{12}t^2+\tfrac{4}{9}t+\tfrac{20}{9},
$$
and so
$$
\frac{I_6(f_{\mu_t})}{I_3(f_{\mu_t})^2}= -\frac{2}{9}\frac{(3t^3+14t^2-64t-160)}{(3t^2+16t+80)^3},
$$
which is easily seen to be an injective function for $t\in [12,\infty)$.  Thus $f_{\mu_t}$, $t\geq 12$, belong to pairwise different projective equivalence classes by Lemma \ref{proyinv}, which implies that $\mu_t$, $t\geq 12$, represents a family of pairwise non-isomorphic nonsingular algebras of type $(3,8)$.  Other intervals for $t$ giving rise to the same conclusion are $(-\infty,-4]$, $[-4,0]$ and $[0,12]$.
\end{example}

\subsection{Positive polynomials}\label{pp}
The following elementary but still intriguing fact follows from the third condition in Definition \ref{ns},
\begin{quote}
$\mu\in V_{n,2d}$ is nonsingular if and only if $f_\mu$ (or $-f_\mu$) is a {\it positive} polynomial, i.e. $f_\mu(x_1,\dots ,x_n)>0$ for any nonzero $(x_1,\dots ,x_n)\in\RR^n$.
\end{quote}
For instance, the $2$-step algebras of type $(3,8)$ given in Examples \ref{444}-\ref{444-pf} and \ref{su2} are both nonsingular.  Notice that one of the forms $\pm f\in P_{n,d}$ is positive if and only if $f$ vanishes only at $0\in\RR^n$, and so for the existence of a nonzero positive $d$-ic form, $d$ must be even.  The subset $P^+_{n,d}\subset P_{n,d}$ of positive forms is open, it is precisely the interior of the closed convex cone of all {\it nonnegative} (i.e. $f\geq 0$) $n$-ary $d$-ic forms (see e.g. \cite[Theorem 3.14]{Rzn1}).

In spite the theory of positive polynomials has a long and rich history in the literature, including Hilbert's 17th Problem (see \cite{Rzn2} for further information), a complete classification of them up to projective equivalence seems to be as difficult as the classification of all forms and thus also unknown and mostly intractable, even in low dimensional cases like ternary quartics (i.e. for the type $(n,m)=(3,8)$).

\begin{example}\label{quadratic}
There is only one positive quadratic form $f\in P_{n,2}$ up to projective equivalence; namely, $f(x_1,\dots,x_n)=x_1^2+\dots+x_n^2$ (use that any projective equivalence class contains an element of the form $\epsilon_1x_1^2+\dots+\epsilon_nx_n^2$ with $\epsilon_i\in\{ 0,\pm 1\}$).  The family $(x_1^2+\dots+x_n^2)^{d/2}$ provides examples for any even degree $d\geq 2$ of positive polynomials whose $\Sl_n$-orbit is closed and are however {\it singular} (i.e. $\tfrac{\partial}{\partial x_i}f=0$, $i=1,\dots,n$ has a nonzero solution in $\CC^n$) for $d\geq 4$.  Their $\Sl_n$-orbits are not of maximal dimension, as their isotropy subgroups all equal $\Or(n)$.
\end{example}

\begin{example}\label{binary}
A complex binary $d$-ic form $f\in P_{2,d}(\CC)$ can be identified (up to scaling) with its set of zeroes in $\PP^1=\PP\CC^2$ as follows:
$$
f(x,y)=\prod_{i=1}^{d} (a_iy-b_ix) \qquad \leftrightarrow\qquad \zca(f)=\left\{ (a_1:b_1),\dots,(a_d:b_d)\right\}\subset\PP^1.
$$
The $\Sl_2(\CC)$-action on the projective space of $P_{2,d}(\CC)$ is therefore equivalent to the action on subsets of unordered $d$ points in $\PP^1$ (counting multiplicities).  The following facts are well known in invariant theory (see e.g. \cite{Dlg,Muk}):
\begin{itemize}
\item $f$ is {\it semistable} (i.e. $0\notin \overline{\Sl_2(\CC)\cdot f}$) if and only if any element in $\zca(f)$ has multiplicity $\leq d/2$.

\item $f$ is {\it stable} (i.e. $\Sl_2(\CC)\cdot f$ is closed and the isotropy subgroup at $f$ is finite) if and only if any element in $\zca(f)$ has multiplicity $< d/2$.
\end{itemize}
It is easy to see that $f$ is real and positive if and only if $d$ is even and
$$
\zca(f)=\{ (1:\alpha_1), (1:\overline{\alpha_1}),\dots,  (1:\alpha_{d/2}), (1:\overline{\alpha_{d/2}})\}, \qquad\mbox{with}\quad \alpha_i\in\CC\setminus\RR.
$$
Positive binary forms are therefore all stable with the only exception of $(x^2+y^2)^{d/2}$, which still has a closed $\Sl_2(\CC)$-orbit.  The $\Sl_2(\CC)$-orbit of a real form $f$ is closed if and only if its $\Sl_2(\RR)$-orbit is so (see e.g. \cite[Proposition 2.3]{BrlHrs} and \cite[Corollary 5.3]{Brk}).  It follows that $\Sl_2(\RR)\cdot f$ is closed for any positive binary form $f$.
\end{example}

\begin{example}\label{ternquart}
It was proved by Hilbert \cite{Hlb} that any nonnegative ternary quartic, as well as any binary or quadratic form, is a sum of squares (Hilbert also showed that this does not hold in any other case).  It is also known that $f\in P_{3,4}$ (or $f$ a binary or quadratic form) is a sum of $4$-th powers of linear forms (and consequently $f\geq 0$) if and only if its Hankel matrix $H(f)$ (see (\ref{calecticant})) is positive semidefinite (see the second sentence after \cite[(5.25)]{Rzn1} and correct a typo by replacing the first $Q$ by $P$).  This is not longer true in the other cases.  However, it is proved in \cite{Ckr} that for any form $f\in P_{n,d}$, we have that if $f>0$ ($f\geq 0$) then $H(f)>0$ ($H(f)\geq 0$).
\end{example}

It is worth noticing at this point that positivity is a notion which only makes sense over $\RR$, and it is therefore difficult to study properties of positive forms from the point of view of invariant theory, a subject mostly attached to algebraically closed fields.  We have seen in Examples \ref{quadratic} and \ref{binary} that positive quadratic and binary forms have closed $\Sl_n$-orbits.  We do not know if this is also true for any form.  It certainly fails in the nonnegative case, e.g. $0\in\overline{\Sl_n\cdot x_1^d}$ for any $n\geq 2$ (and $d$ even).   We now prove that a positive form is at least always semistable.

\begin{lemma}\label{possemi}
If $f\in P_{n,d}$ is positive, then $0\notin\overline{\Sl_n\cdot f}$.
\end{lemma}

\begin{proof}
Assume that $0\in\overline{\Sl_n\cdot f}$, that is, there exists a sequence $\vp_k\in\Sl_n$ such that $f\circ\vp_k\to 0$, as $k\to\infty$.  If $S\subset\RR^n$ is the sphere of radius $1$ and $m:=\min(f|_S)$, which is a positive number due to the positivity of $f$, then
$$
m\|\vp_k(Z)\|^d\leq f(\vp_k(Z)), \qquad \forall Z\in\ngo_2.
$$
This implies that
$$
m\|\vp_k\|^d \leq\max(f\circ\vp_k|_S)\to 0,
$$
as $k\to\infty$, and thus $\vp_k\to 0$, which is a contradiction since $\det{\vp_k}=1$ for all $k$.  This concludes the proof of the lemma.
\end{proof}

\subsection{Real geometric invariant theory}\label{git}
Let $G$ be a real reductive group acting linearly on a finite dimensional real
vector space $V$. The precise definition of
our setting is the one considered in \cite{RchSld}, or the more general one in \cite{HnzSchStt} (see also \cite{EbrJbl}),
where many results from geometric invariant theory are adapted and proved over $\RR$.

The derivative of the
above action defines a representation of the Lie algebra $\ggo$ of $G$ in $V$, which will be denoted by $\pi:\ggo\longrightarrow\End(V)$. We consider
a Cartan decomposition $\ggo=\kg\oplus\pg$, where $\kg$ is the Lie algebra of a maximal compact subgroup
$K$ of $G$. Endow $V$ with a fixed from now on $K$-invariant inner product $\ip$
such that $\pg$ acts on $V$ by symmetric operators, and endow $\pg$ with an
$\Ad(K)$-invariant inner product also denoted by $\ip$.

The function $m:V\longrightarrow\pg$ implicitly defined by
$$
\la m(v),\alpha\ra=\isn\la\pi(\alpha)v,v\ra, \qquad \forall\alpha\in\pg,\; v\in V\smallsetminus\{ 0\}, \qquad m(0)=0,
$$
is called the {\it moment map} for the representation $V$ of $G$.  Since
$m(cv)=m(v)$ for any nonzero $c\in\RR$, we may also view the moment map as defined on the
projective space $\PP V$ of $V$.  It is easy to see that
$m$ is $K$-equivariant: $m(k.v)=\Ad(k)m(v)$ for all $k\in K$.

Let $\mca=\mca(G,V)$ denote the set of {\it minimal vectors}, that is,
$$
\mca=\{ v\in V: ||v||\leq ||g.v||\quad\forall g\in G\}.
$$
In \cite{RchSld}, the following results were obtained:

\begin{itemize}
\item an orbit $G.v$ is closed if and only if $G.v$ meets $\mca$;

\item $G.v\cap\mca$ is either empty or consists of a single $K$-orbit;

\item the closure of any $G$-orbit contains a unique closed $G$-orbit;

\item $\mca = \{ v\in V: m(v)=0\}$.
\end{itemize}

The functional square norm of the moment map
\begin{equation}\label{norm}
F_m:V\longrightarrow\RR, \qquad  F_m(v):=||m(v)||^2,
\end{equation}
is scaling invariant and so it can actually be viewed as a function on any sphere
of $V$ or on $\PP V$.  The remaining critical points
of $F_m$ other than minimal vectors (i.e. those for which $F_m(v)>0$) are all {\it unstable} (i.e. $0\in\overline{G.v}$) but still enjoy
most of the nice properties of minimal vectors stated above.  Let $\cca(F_m)$ denote the critical set of $F_m:V\longrightarrow\RR$.

\begin{theorem}\cite{Mrn,HnzSchStt}\label{crit}
The following conditions are equivalent:
\begin{itemize}
\item[(i)] $v\in\cca(F_m)$.

\item[(ii)] The functional $F_m|_{G.v}$ attains its minimum value at $v$.

\item[(iii)] $\pi(m(v))v=cv$ for some $c\in\RR$.
\end{itemize}
Moreover, the following uniqueness property holds:
\begin{itemize}
\item[(iv)] The intersection of $\cca(F_m)$ with any $G$-orbit is either empty or consists of a single $K$-orbit (up to scaling).
\end{itemize}
\end{theorem}

It follows from part (iv) that two $G$-orbits containing critical points of $F_m$ can be distinguished by using $K$-invariants, which are always much more abundant.  This method will be illustrated in the examples below.

\begin{example}\label{hompol}
Let us consider the natural linear action of $G=\Gl_n$ on $V=P_{n,d}$ given by $\vp\cdot f=f\circ\vp^{-1}$.  We have that $\ggo=\glg_n$, $K=\Or(n)$, $\kg=\sog(n)$ and $\pg=\sym(n)$.  As an $\Ad(K)$-invariant
inner product on $\pg$ we take $\la\alpha,\beta\ra=\tr{\alpha\beta}$, and it is easy to
see that the inner product $\ip$ on $V$ for which the basis of monomials
$$
\left\{ x^D:=x_1^{d_1}\dots x_n^{d_n}:d_1+\dots +d_n=d,\; D=(d_1,\dots,d_n)\right\}
$$
is orthogonal and
$$
||x^D||^2=\tfrac{d_1!\dots d_n!}{d!}, \qquad\forall D=(d_1,\dots ,d_n),
$$
satisfies the required conditions (indeed, $\la f\circ\alpha,g\ra=\la f,g\circ\alpha^t\ra$ for all $f,g\in P_{n,d}$, $\alpha\in\glg_n$, see e.g. \cite[(1.30)]{Rzn1}).  Let $E_{ij}$ denote as usual the $n\times n$
matrix whose only nonzero coefficient is a $1$ in the entries $ij$.  Since
$$
\pi(E_{ij})f=\ddt|_{t=0} f\circ e^{-tE_{ij}}=-x_j\tfrac{\partial f}{\partial x_i},
$$
we obtain that the moment map $m:P_{n,d}\longrightarrow\sym(n)$ is given by
\begin{equation}\label{mmforms}
m(f)=-\tfrac{1}{||f||^2}\left[\la x_j\tfrac{\partial f}{\partial x_i},f\ra\right].
\end{equation}
We are using here the fact that $\la x_j\tfrac{\partial f}{\partial x_i},f\ra=\la
x_i\tfrac{\partial f}{\partial x_j},f\ra$ for all $i,j$.  It is also easy to see that the action of a diagonal matrix $\alpha\in\glg_n$
with entries $a_1,\dots ,a_n$ is given by
\begin{equation}\label{diagact}
\pi(\alpha)x^D=-\left(\sum_{i=1}^na_id_i\right)x^D, \qquad\forall D=(d_1,\dots ,d_n),
\end{equation}
and since
$$
m(x^D)=\left[\begin{smallmatrix} -d_1&&\\ &\ddots&\\ &&-d_n
\end{smallmatrix}\right],
$$
we get that $x^D$ is an eigenvector of $m(x^D)$ with eigenvalue $F_m(x^D)=\sum d_i^2$.  Every monomial is therefore a critical point of $F_m$ by Theorem \ref{crit}, although $\cca(F_m)$ is known to be much larger (see \cite[Section 10]{Nss} and \cite[Example 11.5]{cruzchica}).
\end{example}

According to the above example, it follows from Theorem \ref{crit}, (iv) that $\Or(n)$-invariants of $n$-ary forms play a key role in distinguishing critical points of $F_m:P_{n,d}\longrightarrow\RR$.  Let us therefore consider the {\it Laplacian} $\Delta :P_{n,d}\longrightarrow P_{n,d-2}$ defined by
$$
\Delta(f)  :=  \frac{\partial^2f}{\partial x_1^2}+ \dots  +\frac{\partial^2f}{\partial x_n^2},
$$
which is well known to be $\Or(n)$-equivariant:
$$
\Delta(\vp\cdot f)=\vp\cdot\Delta(f), \qquad\forall \vp\in\Or(n),\quad f\in P_{n,d}.
$$

\begin{example}\label{1-parlap}
Is the curve of ternary quartics $f_t=x^4+y^4+z^4+tx^2y^2$ pairwise non-equivalent?  The argument used in Example \ref{exinvform} in terms of $\Sl_3$-invariants does not work here because $I_6(f_t)=0$ for all $t$.  However, if we consider
$$
\vp_t\cdot f_t= x^4+y^4+(1+t^2/12)^{1/2}z^4+tx^2y^2, \qquad\mbox{where}\quad \vp_t:=\left[
    \begin{smallmatrix}
      1 &  &  \\
       & 1 &  \\
       &  & (1+t^2/12)^{-1/8} \\
    \end{smallmatrix}
  \right],
$$
then it is straightforward to check by using (\ref{mmforms}) that $m(\vp_t\cdot f_t)=-\tfrac{4}{3}I$, which implies that $\vp_t\cdot f_t\in\cca(F_m)$ for all $t$ by Theorem \ref{crit}, (iii).  According to Theorem \ref{crit}, (iv) we now only need to differentiate $\vp_t\cdot f_t$ with $\Or(3)$-invariants (up to scaling).  A computation of the Laplacian gives
$$
\Delta(\vp_t\cdot f_t)  = (12+2t)x^2+(12+2t)y^2+12(1+t^2/12)^{1/2}z^{2}.
$$
The function $\frac{12+2t}{12(1+t^2/12)^{1/2}}$ is the quotient of the two eigenvalues of these ternary quadratic forms,
and since it is a strictly decreasing function for $t\geq 2$,
it follows that $\vp_t\cdot f_t$ (or equivalently, $f_t$), $t\geq 2$, is a family of pairwise non-equivalent ternary quartics.
\end{example}

We now consider the representation where $2$-step algebras live.

\begin{example}\label{brackets}
Let $V=V_{n,m}$ be the representation of the real reductive group $G=\Gl_m\times\Gl_n$ given by (\ref{actionG}), for which we have $\ggo=\glg_m\oplus\glg_n$, $K=\Or(m)\times\Or(n)$, $\kg=\sog(m)\oplus\sog(n)$ and $\pg=\sym(m)\oplus\sym(n)$.  Recall that we have fixed an inner product $\ip$ on $\ngo$.  As an $\Ad(K)$-invariant
inner product on $\pg$ we can take $\la\alpha,\beta\ra=\tr{\alpha\beta}$ on each factor and the inner product $\ip$ on $V$ is defined by
\begin{equation}\label{innV}
\la\mu,\lambda\ra:= \sum\limits_{ij}\la\mu(X_i,X_j),\lambda(X_i,X_j)\ra
=\sum\limits_{ijk} \mu_{ij}^k\lambda_{ij}^k.
\end{equation}
The corresponding representation of $\ggo$ on $V$ is
\begin{equation}\label{actiong}
\pi(\alpha,\beta)\mu=\beta\mu(\cdot,\cdot)-\mu(\alpha\cdot,\cdot)-\mu(\cdot,\alpha\cdot),
\end{equation}
and the moment map $m:V_{n,m}\longrightarrow\sym(m)\oplus\sym(n)$ is given by $m(\mu)=(m_1(\mu),m_2(\mu))$ (see e.g. \cite{Ebr2}), where
\begin{equation}\label{mmalg}
m_1(\mu)=\tfrac{2}{\|\mu\|^2}\sum_i J_\mu(Z_i)^2, \qquad \la m_2(\mu)Z,W\ra=-\tfrac{1}{\|\mu\|^2}\tr{J_\mu(Z)J_\mu(W)}, \quad \forall Z,W\in\ngo_2.
\end{equation}
\end{example}

\subsection{Nilsolitons}\label{nil}
We consider in this subsection the problem of the existence of a `canonical' inner product for a given nonsingular algebra.  Our approach will follow the lines of real geometric invariant theory described in Section \ref{git}, and it turns out that the uniqueness (up to isometry) of such distinguished inner product will give rise to new invariants.

Recall the maps $J_\mu(Z)\in\sog(\ngo_1)$, $Z\in\ngo_2$, defined at the beginning of Section \ref{inv} by fixing an inner product $\ip$ on $\ngo=\ngo_1\oplus\ngo_2$ (with $\ngo_1\perp\ngo_2$).  These maps provide a nice tool to consider compatibility conditions between an algebra $\mu$ and an inner product $\ip$.  For instance, when $J_\mu(Z)^2=-\| Z\|^2I$ for any $Z\in\ngo_2$, the metric Lie algebra $(\ngo,\mu,\ip)$ is called {\it H-type} (see \cite{Kpl,BrnTrcVnh}), and if more generally, $Z\mapsto(-\det{J_\mu(Z)^2})^{1/m}$ is a positive quadratic form, then the algebra is said to be of {\it $\tilde{{\text H}}$-type} (see \cite{LvsTrb}).  Equivalently, in terms of its Pfaffian form, $(\ngo,\mu,\ip)$ is $\tilde{{\text H}}$-type if and only if, up to projective equivalence,
$$
f_\mu(x_1,\dots,x_n)=(x_1^2+\dots +x_n^2)^{m/4}.
$$
For instance, the algebra in Example \ref{su2} is of $\tilde{{\text H}}$-type but not of $H$-type.

We may also consider the following condition involving only $\mu$ and $\ip$, as another generalization of H-type: for any (or some) orthonormal basis $\{ Z_i\}$ of $\ngo_2$,
\begin{equation}\label{nilso}
\sum_i J_\mu(Z_i)^2=aI, \qquad \tr{J_\mu(Z_i)J_\mu(Z_j)}=b\delta_{ij}, \qquad \mbox{for some} \quad a,b<0.
\end{equation}
This condition is called {\it optimal Ricci tensor} in \cite[Section 7]{Ebr2} and it is proved in \cite[Section 7]{Ebr4} that if the condition on the left in \eqref{nilso} holds, then the metric has a geodesic-flow invariant Ricci tensor.  It has also been studied in \cite[Section 9]{inter}.

In terms of the moment map (see Example \ref{brackets}), condition \eqref{nilso} on $(\mu,\ip)$ is equivalent to have $m(\mu)=\left(2aI,-bI\right)$ (see (\ref{mmalg})), from which easily follows that part (iii) in Theorem \ref{crit} holds and thus $\mu$ is a critical point of the functional $F_m(\lambda)=\|m(\lambda)\|^2$.  From the uniqueness statement in part (iv) of the same theorem, we obtain that among the isomorphism class $\Gl_m\times\Gl_n\cdot\mu$, only the elements in the subset $\RR^*\Or(n)\cdot\mu$ satisfy condition (\ref{nilso}).  By using that such condition holds for the pair $(\vp\cdot\mu,\ip)$, $\vp\in\Gl_m\times\Gl_n$ if and only if does so for $(\mu,\la\vp\cdot,\vp\cdot\ra)$, this uniqueness property can be rephrased as follows:

\begin{quote}
Given an algebra $\mu\in V_{n,m}$, there exists at most one inner product on $\ngo$ (with $\ngo_1\perp\ngo_2$) for which (\ref{nilso}) holds, up to scaling and automorphisms of $\mu$.
\end{quote}

This strongly supports the proposal of these inner products as canonical or distinguished for a given $\mu$.

It also follows from invariant theory that if condition (\ref{nilso}) holds, then the $\Sl_m\times\Sl_n$-orbit of $\mu$ is closed in $V_{n,m}$.  Indeed, the orthogonal projection of the moment map $m(\mu)$ onto $\slg_m\oplus\slg_n$ vanishes, and since this is precisely the moment map for the $\Sl_m\times\Sl_n$-action on $V_{n,m}$ we get that $\mu$ is a minimal vector and so $\Sl_m\times\Sl_n\cdot\mu$ is closed (see Section \ref{git}).  Furthermore, the following conditions are equivalent for a $2$-step algebra $\mu\in V_{n,m}$:

\begin{itemize}
\item there exists an inner product $\ip$ on $\ngo$ such that condition (\ref{nilso}) holds for $(\mu,\ip)$ (in particular, $\mu\in\cca(F_m)$);

\item for any fixed inner product $\ip$ on $\ngo$, there exists an isomorphic algebra $\mu_0\in\Gl_m\times\Gl_n\cdot\mu$ such that condition (\ref{nilso}) holds for $(\mu_0,\ip)$ (or equivalently, $\mu_0$ is a minimal vector for the $\Sl_m\times\Sl_n$-action);

\item the orbit $\Sl_m\times\Sl_n\cdot\mu$ is closed.
\end{itemize}

In particular, the existence for a given nonsingular $\mu$ of an inner product satisfying (\ref{nilso}) provides in a way a candidate for a `normal form' for $\mu$; namely, the minimal vector $\mu_0$ above, which is unique up to scaling and the action of $\Or(m)\times\Or(n)$.

From a geometric point of view, we know that $\mu$ is a critical point of $F_m$ if and only if the left-invariant Riemannian metric $g$ determined by $(\ngo,\mu,\ip)$ on the corresponding simply connected nilpotent Lie group $N$ is a {\it Ricci soliton}, i.e. the Ricci flow solution starting at $g$ evolves only by time-dependent scaling and pull-back by diffeomorphisms (see \cite{soliton}).  Such metrics are called {\it nilsolitons} in the literature and have been extensively studied in the last decade (see the survey \cite{cruzchica} for further information).  The core of this interplay relies on the fact that
\begin{equation}\label{mmric}
m(\mu)=\tfrac{4}{\|\mu\|^2}\Ricci(g),
\end{equation}
where $m(\mu)$ is the moment map and $\Ricci(g)$ is the Ricci operator of $(N,g)$.

We now prove that condition (\ref{nilso}) must actually hold for any nilsoliton on a nonsingular algebra.

\begin{theorem}\label{nilthm}
Let $\mu$ be a nonsingular algebra of type $(n,m)$ and set $\ngo=\ngo_1\oplus\ngo_2$. Then the following conditions are equivalent:
\begin{itemize}
\item[(i)] $\ngo$ admits an inner product for which condition {\rm (\ref{nilso})} holds.

\item[(ii)] The Lie algebra $(\ngo,\mu)$ admits a nilsoliton inner product (or equivalently, $\mu\in\cca(F_m)$ for some $\ip$ on $\ngo$).

\item[(iii)] The $\Sl_m\times\Sl_n$-orbit of $\mu$ is closed in $V_{n,m}$.
\end{itemize}
\end{theorem}

\begin{remark}
This theorem also follows by using \cite[Lemma 5]{Nkl1} and \cite[Theorem 2]{Nkl2}.
\end{remark}

\begin{remark}
The equivalence between parts (i) and (iii) has been observed above and does not need the hypothesis of nonsingularity to hold (cf. \cite[Proposition 9.1]{inter} or \cite[Proposition 7.4]{Ebr2}).  The fact that part (i) implies (ii) has also been mentioned.  It is worthwhile pointing out that these conditions were proved in \cite[Proposition 7.9]{Ebr2}, and independently in \cite{Nkl2}, to hold on an open and dense subset of $V_{n,m}$ if $(m,n)\ne (2k-1,2)$ for any $k\in\NN$.
\end{remark}

\begin{proof}
It only remains to prove that part (ii) implies part (i).  It follows from Lemma \ref{possemi} that the nonsingularity of $\mu$ implies that
\begin{equation}\label{ns-ss}
0\notin\overline{\Sl_m\times\Sl_n\cdot\mu},
 \end{equation}
the closure of the orbit relative to the vector space topology. Indeed, if $\mu_k:=(\psi_k,\vp_k)\cdot\mu\to 0$, as $k\to\infty$, with $(\psi_k,\vp_k)\in\Sl_m\times\Sl_n$, then their Pfaffian forms $f_{\mu_k}\to 0$, which are given by $f_{\mu_k}=f_\mu\circ\vp_k^t\in\Sl_n\cdot f_\mu$ for all $k$, and so this contradicts Lemma \ref{possemi} since $f_\mu$ is positive.

If we use the nilsoliton inner product on $\ngo$ (with $\ngo_1\perp\ngo_2$) to define the moment map $m$, then $\mu\in\cca(F_m)$ by \cite[Theorem 4.2]{soliton}.
In what follows, we will use the notation in \cite[pp. 1869-1870]{standard}.  For
$$
\beta=\left[\begin{smallmatrix} -\tfrac{2}{m}I_m & \\ & \tfrac{1}{n}I_n\end{smallmatrix}\right]\in\glg_{m+n},
$$
it is straightforward to check by using \eqref{actiong} that $\pi(\beta)\mu=\|\beta\|^2\mu$, which implies that $\mu\in Z_\beta$ (see the first paragraph after the definition of $Z_\beta$ in \cite[pp.1869]{standard}).  Since $\beta\perp D$, where
$$
\beta+\|\beta\|^2I=(2m+\tfrac{1}{n})D\in\Der(\mu), \qquad D:=
\left[\begin{smallmatrix} I_m & \\ & 2I_n\end{smallmatrix}\right],
$$
one obtains that $\hg_\beta=\slg_m\oplus\slg_n\oplus\RR D$, and so $H_\beta\cdot\mu=\Sl_m\times\Sl_n\cdot\mu$.  It now follows from (\ref{ns-ss}) and \cite[Proposition 2.14,(i)]{standard} that $\mu$ belongs to the stratum $\sca_\beta$.

Finally, by using \cite[Theorem 2.3]{einsteinsolv} and \cite[Proposition 3.4,(ii)]{einsteinsolv}, we get from $\mu\in\cca(F_m)$ that $m(\mu)$ is conjugate to $\beta$.  It follows from (\ref{mmric}) that the Ricci operator $\Ricci(g)$ has exactly two nonzero eigenvalues of opposite signs.  By using the formula for $\Ricci(g)$ given in \cite[Proposition 2.5]{Ebr}, we obtain that condition (\ref{nilso}) holds for the pair $(\vp\cdot\mu,\ip)$ for some $\vp\in\Or(m)\times\Or(n)$, and hence it does so for $(\mu,\ip)$, from which part (i) follows.
\end{proof}

By fixing an inner product $\ip$ on $\ngo$ as in previous sections, we deduce the following from Theorem \ref{nilthm}.

\begin{corollary}\label{nilcor}
If $\mu$ is nonsingular, then the following conditions are equivalent:
\begin{itemize}
\item [(i)] The nilmanifold $(N,g)$ attached to $(\ngo,\mu,\ip)$ is a nilsoliton.

\item[(ii)] Condition {\rm (\ref{nilso})} holds for $(\ngo,\mu,\ip)$.

\item[(iii)] The Ricci operator $\Ricci(g)$ has exactly two different eigenvalues.
\end{itemize}
\end{corollary}

\begin{example}
$H$-type Lie groups are all nonsingular nilsolitons.  On the other hand, it is easy to check that the $2$-step algebra of type $(3,8)$ given in Example \ref{su2} is a nonsingular nilsoliton which is not $H$-type.
\end{example}

In order to get explicit continuous families of nonsingular algebras of type $(n,m)$ admitting a nilsoliton metric, we may consider those $\mu\in V_{n,m}$ such that each entry of $J_\mu(x_1Z_1+\dots+x_nZ_n)$ is a scalar multiple of some $x_i$ and there is not any repetition of the $x_i$'s appearing in any given row (or column).  This is equivalent to say that $\{ X_1,\dots,X_m,Z_1,\dots,Z_n\}$ is a {\it nice} basis for the algebra $\mu$
(see e.g. \cite[Lemma 3.9]{einsteinsolv} or \cite[Definition 3]{Nkl2}).  In the presence of a nice basis, Y. Nikolayevsky has proved a very useful criterion for the existence of a nilsoliton; namely, when a certain symmetric $q\times q$ matrix $U$ which only depends on the set of indices $(i,j,k)$ with $\mu_{ij}^k\ne 0$, has a positive solution $v\in\RR^q$ to the equation $Uv=[1]$, where $[1]\in\RR^q$ denotes the vector with all its entries $1$ (see \cite[Theorem 3]{Nkl2}).  This approach will be used in Section \ref{38}.

If in addition to the nice basis condition we ask for each row (or column) of $J_\mu(x_1Z_1+\dots+x_nZ_n)$ to have the same number of nonzero entries, then the metric algebra is called {\it uniform} (see \cite{Dlf,Wlt}).  In this case, it is easy to see that the matrix $U$ from Nikolayevky's criterion has all $3$'s as diagonal entries (as always) and only $1$'s out of the diagonal with an identical number of $1$'s on each row (or column).  It follows that there exists a positive $v\in\RR^q$ to $Uv=[1]$ (just take $v_1=\dots=v_q=a$ for the right $a$), and thus $\mu$ admits a nilsoliton inner product.  It is important to note that one obtains this independently of the precise value of the structure constants $\mu_{ij}^k$ of the algebra $\mu$, as the uniformity condition is only in terms of the set of structure constants which are nonzero.

Let us summarize the above discussion in a lemma for future reference.

\begin{lemma}\label{uniform}\cite{Dlf,Wlt}
Any uniform $\mu\in V_{n,m}$ admits a nilsoliton inner product.
\end{lemma}

In particular, according to Theorem \ref{nilthm}, any uniform nonsingular algebra admits an inner product for which condition (\ref{nilso}) holds.

\section{Nonsingular algebras of type $(2,m)$}\label{2m}

The set of isomorphism classes in $V_{2,m}$, that is, when $\dim{\ngo_2}=2$, can be nicely parameterized by using pencil invariants.  We refer to \cite{Ggr,LvsTrb} and references therein for more detailed treatments and proofs.

To each set
$$
\sca=\{ (\alpha_1,k_1), \dots  ,(\alpha_r,k_r), \epsilon_1, \dots ,\epsilon_s\},  \qquad\alpha_i\in\CC\cup\{\infty\}, \quad k_i,\epsilon_i\in\NN,
$$
we associate the element $\mu_{\sca}\in V_{2,m}$ of type $(2,m)$ such that the matrix of $J_{\mu_{\sca}}(xZ_1+yZ_2)$ with respect to the fixed basis $\{ X_1,\dots,X_m\}$ is made of blocks attached to each of the elements of $\sca$, which are defined according to the following rules:

{\small
$$
(\infty,k)\rightsquigarrow \left[
  \begin{array}{ccc|cccc}
       &  &  & 0 &  &  & x \\
       &  &  &  &  & x & y \\
       &  &  &  & \udots & \udots &  \\
       &  &  & x & y &  & 0 \\
     \hline
     & & &  &  &  &  \\
     & \ast &  &  &  &  &  \\
     &  &  &  &  &  &  \\
  \end{array}
\right] \quad(2k\times 2k),
$$

$$
(a,k)\rightsquigarrow  \left[
  \begin{array}{ccc|crrc}
       &&  & 0 & &  & y-ax \\
       &&  & &  & y-ax & x \\
       &&  &  & \udots & \udots &  \\
       &&  & y-ax & x & & 0 \\
     \hline
     &&  &&&&  \\
     & \ast &    &&&&  \\
     &&  &&&&  \\
  \end{array}
\right]  \quad   (2k\times 2k),
$$

\begin{equation}\label{Jcomplex}
(a+\im b,k)\rightsquigarrow  \left[
\begin{array}{ccc|cccccccc}
    &&    & 0&  &  &  &  &  & -bx & y-ax \\
    &&    & &  &  &  &  &  & y-ax & bx \\
    &&    & & &  &  & -bx & y-ax & 0 & x \\
    &&    & & &  &  & y-ax & bx & x & 0 \\
    &&    & & &  &  &  &  &  &  \\
    &&    & & & \udots &  & & \udots  &  &  \\
    &&    & & &  &  &  &  &  &  \\
    &&  &-bx& y-ax & 0 & x &  &  &  &  \\
    &&  &y-ax & bx & x & 0 &  &  &  & 0 \\
   \hline
   &&  &&&&&&&& \\
   & \ast &  &&&&&&&&  \\
   &&    &&&&&&&&
\end{array}
\right] \quad (4k\times 4k),
\end{equation}

$$
\epsilon\rightsquigarrow  \left[
  \begin{array}{ccc|cccc}
&& &y&&&0\\
&& &x&y&&\\
&& &&\ddots&\ddots&\\
&& &&&x&y\\
&& &0&&&x\\
\hline
&& &&&&\\
&\ast& &&&&\\
&& &&&&\\
  \end{array}
\right] \quad ((2\epsilon+1)\times (2\epsilon+1)).
$$ }
The upper right block in the first two matrices above is $k\times k$, in the third one is $2k\times 2k$ and in the last matrix has $\epsilon+1$ rows and $\epsilon$ columns (the size of the whole matrix is indicated on the right of each matrix).  Notice that
$$
m=\sum_{\alpha_i\in\RR\cup\{\infty\}} 2k_i+ \sum_{\alpha_i\in\CC\smallsetminus\RR} 4k_i+\sum_{i=1}^s(2\epsilon_i+1).
$$
Every $\mu\in V_{2,m}$ of type $(2,m)$ (i.e. $\mu(\ngo_1,\ngo_1)=\ngo_2$) is isomorphic to a $\mu_{\sca}$ for at least one of such sets $\sca$, and two $\mu_{\sca}, \mu_{\sca'}$ are isomorphic if and only if $s=s'$, $\{\epsilon_1, \dots ,\epsilon_s\}=\{\epsilon'_1, \dots ,\epsilon'_s\}$ (counting multiplicities), $r=r'$, and there exists a real M\"obius transformation $T:\CC\cup\{\infty\}\longrightarrow\CC\cup\{\infty\}$,
\begin{equation}\label{isomuS}
Tz:=\frac{az+b}{cz+d}, \qquad \left[\begin{array}{cc} a&b\\ c&d\end{array}\right]\in\Gl_2(\RR),
\end{equation}
such that $\{ (\alpha'_1,k'_1), \dots  ,(\alpha'_r,k'_r)\}= \{(T\alpha_1,k_1), \dots  ,(T\alpha_r,k_r)\}$ (counting multiplicities).

It is straightforward to see that the Pfaffian form of $\mu_{\sca}$ is given by
\begin{equation}\label{PfmuS}
f_{\mu_{\sca}}(x,y)=x^{\left(\sum\limits_{\alpha_i=\infty} k_i\right)}\, .\, \prod_{\alpha_i\in\RR} (y-\alpha_ix)^{k_i} \, .\, \prod_{\alpha_i\in\CC\smallsetminus\RR} (y-\alpha_ix)^{k_i}(y-\overline{\alpha_i}x)^{k_i},
\end{equation}
from which we deduce that $\mu_{\sca}$ is nonsingular if and only if
$$
\sca=\{ (\alpha_1,k_1), \dots  ,(\alpha_r,k_r)\},  \qquad \mbox{with} \quad\alpha_1,\dots ,\alpha_r\in\CC\smallsetminus\RR.
$$
In that case, if $\alpha_i=a_i+\im b_i$ then
$$
f_{\mu_{\sca}}(x,y)=\prod_{i} \left((y-a_ix)^2+b_i^2x^2\right)^{k_i}.
$$

\begin{example}
If $\sca=\{(\im,1),(\im,1)\}$ and $\sca'=\{(\im,2)\}$ then $\mu_{\sca},\mu_{\sca'}\in V_{2,8}$ are two non-isomorphic nonsingular algebras with identical Pfaffian form given by $(x^2+y^2)^2$.  By adding the family $\mu_{\sca_t}$, $\sca_t=\{(\im,1),(t\im,1)\}$, $t>1$, of pairwise non-isomorphic algebras with Pfaffian forms $f_{\mu_{\sca_t}}=(x^2+y^2)(t^2x^2+y^2)$, one obtains the classification of all nonsingular algebras of type $(2,8)$ (cf. \cite[Corollary 4.9,(ii)]{LvsTrb}).  The non-isomorphism within the family follows by using the fact that the hyperbolic distance between $\im$ and $t\im$ strictly increases with $t$, or alternatively, from the non-equivalence between their Pfaffian forms.
\end{example}

Y. Nikolayevsky has obtained a complete classification of all algebras of type $(2,m)$ admitting a nilsoliton inner product.

\begin{theorem}\label{nilYN}\cite{Nkl3}
Let $\mu_{\sca}$ be a $2$-step algebra of type $(2,m)$, where
$$
\begin{array}{c}
\sca=\{ (\alpha_1,k_1), \dots, (\alpha_r,k_r), (a_1,j_1), \dots  ,(a_u,j_u), \epsilon_1, \dots ,\epsilon_s\}, \\
\alpha_i\in\CC\smallsetminus\RR, \quad a_i\in\RR\cup\{\infty\}, \quad \epsilon_i,k_i,j_i\in\NN,
\end{array}
$$
and assume that either $r>0$ or $u\geq 3$.  Then $\mu_{\sca}$ admits a nilsoliton inner product if and only if
\begin{itemize}
\item $k_1=\dots =k_r=1$ and $j_1=\dots =j_u=1$, and
\item[ ]
\item $|\{i:a_i=a\}|<r+\unm u$, for any $a\in\RR\cup\{\infty\}$.
\end{itemize}
\end{theorem}

We have only stated a part of the classification theorem \cite[Theorem 1]{Nkl3}, called the generic case, as the remaining parts are hard to describe and are not needed in the nonsingular case.

\begin{corollary}\label{nil2m}
A nonsingular $\mu_{\sca}$, $\sca=\{ (\alpha_1,k_1), \dots  ,(\alpha_r,k_r)\}$, admits a nilsoliton inner product if and only if $k_1=\dots =k_r=1$.
\end{corollary}

We use the next two results to provide an alternative, much shorter proof of this corollary as an application of Theorem \ref{nilthm}.  According to such theorem, the problem of which nonsingular algebras of type $(2,m)$ admit a nilsoliton inner product can be solved by determining which $\Sl_m\times\Sl_2$-orbits are closed.

\begin{lemma}\label{deg}
If $\sca=\{ (\alpha,k)\}$ and $\sca'=\{(\alpha,1),\dots,(\alpha,1)\}$ ($k$ times), $\alpha\in\CC\setminus\RR$, then
$$
\mu_{\sca'}\in\overline{\Sl_{4k}\times\{ I_2\}\cdot\mu_{\sca}},
$$
where $I_2$ denotes the $2\times 2$ identity matrix.
\end{lemma}

\begin{proof}
Recall that the matrix $J_{\mu_{\sca}}(xZ_1+yZ_2)$ defining $\mu_{\sca}$ is made of a single $4k\times 4k$ block as in (\ref{Jcomplex}) for $\alpha=a+\im b$.  If we consider $\mu_t:=\vp_t\cdot\mu_{\sca}$, where
$$
\vp_t:=\Diag(e^{t},e^{t},e^{-t},e^{-t},\dots,e^{t},e^{t},e^{-t},e^{-t},1,1)\in\Sl_{4k}\times\{ I_2\},
$$
then it is easy to check that the upper right block of $J_{\mu_t}(xZ_1+yZ_2)$ is given for all $t$ by {\small
$$
\left[
\begin{array}{cccccccc}
         0&  &  &  &  &  & -bx & y-ax \\
         &  &  &  &  &  & y-ax & bx \\
         & &  &  & -bx & y-ax & 0 & e^{-2t}x \\
         & &  &  & y-ax & bx & e^{-2t}x & 0 \\
         & &  &  &  &  &  &  \\
         & & \udots &  & & \udots  &  &  \\
         & &  &  &  &  &  &  \\
      -bx& y-ax & 0 & e^{-2t}x &  &  &  &  \\
      y-ax & bx & e^{-2t}x & 0 &  &  &  & 0
   \end{array}
\right]
$$}
This implies that $\mu_t\to\mu_{\sca'}$, as $t\to\infty$, since in the limit we get a number of $k$ $4\times 4$ blocks of the form (\ref{Jcomplex}), up to permutation.  We therefore obtain that $\mu_{\sca'}\in\overline{\Sl_{4k}\times\{ I_2\}\cdot\mu_{\sca}}$, as was to be shown.
\end{proof}

The above proof can be easily modified in order to get the following more general result.

\begin{corollary}\label{degcor}
$\mu_{\sca'}\in\overline{\Sl_{4k}\times\{ I_2\}\cdot\mu_{\sca}}$ for any $\sca=\{(\alpha_1,k_1),\dots,(\alpha_r,k_r)\}$, $\alpha_i\in\CC\setminus\RR$, and $\sca'=\{(\alpha_1,j_{11}),\dots,(\alpha_1,j_{1s_1}),\dots,(\alpha_r,j_{r1}),\dots,(\alpha_r,j_{rs_r})\}$ such that
$$
\sum_{i=1}^{s_1} j_{1i} = k_1, \quad\dots\quad, \sum_{i=1}^{s_r} j_{ri} = k_r.
$$
\end{corollary}

\begin{proof}[Proof of Corollary \ref{nil2m}]
If $k_i\geq 2$ for some $i$, then it follows from Corollary \ref{degcor} that the orbit $\Sl_m\times\Sl_2\cdot\mu_{\sca}$ is not closed, and thus $\mu_{\sca}$ does not admit a nilsoliton inner product by Theorem \ref{nilthm}.

Conversely, let $\mu_{\sca}\in V_{2,4r}$ be a nonsingular algebra with $\sca=\{(\alpha_1,1),\dots,(\alpha_r,1)\}$, $\alpha_i\in\CC\setminus\RR$.  According to Theorem \ref{nilthm}, it is enough to show that the orbit $\Sl_{4r}\times\Sl_2\cdot\mu_{\sca}$ is closed in $V_{2,4r}$.  Assume that $\mu_k:=(\psi_k,\vp_k)\cdot\mu_{\sca}\to\lambda\in V_{2,4r}$, as $k\to\infty$, with $(\psi_k,\vp_k)\in\Sl_{4r}\times\Sl_2$.  Thus their Pfaffian forms $f_{\mu_k}\to f_\lambda$ and are given by $f_{\mu_k}=f_{\mu_{\sca}}\circ\vp_k^t\in\Sl_2\cdot f_{\mu_{\sca}}$ for all $k$, from which it follows that $f_\lambda\in\Sl_2\cdot f_{\mu_{\sca}}$ since the $\Sl_2$-orbit of any positive binary form is closed (see the discussion following Example \ref{binary}).

In particular, $\lambda$ is also nonsingular.  Moreover, it follows from \eqref{PfmuS} and the equivalence between $f_\lambda$ and $f_{\mu_{\sca}}=\prod_{i=1}^r(y-\alpha_ix)(y-\overline{\alpha_i}x)$ that $\lambda$ must be isomorphic to an algebra $\mu_{\sca'}$ for some set $\sca'$ of the form
$$
\sca'=\{(\beta_1,k_1),\dots,(\beta_s,k_s)\},
$$
such that for any fixed $j=1,\dots,r$,
$$
\sum_{\beta_i=\alpha_j} k_i=|\{\alpha_i:\alpha_i=\alpha_j\}|.
$$
It is easy to see by using Corollary \ref{degcor} that $\mu_{\sca}\in\overline{\Sl_{4r}\times\Sl_2\cdot\lambda}$, and since $\lambda\in\overline{\Sl_{4r}\times\Sl_2\cdot\mu_{\sca}}$, we obtain that $\lambda\in\Sl_{4r}\times\Sl_2\cdot\mu_{\sca}$.  This implies that $\Sl_{4r}\times\Sl_2\cdot\mu_{\sca}$ is indeed closed, as was to be shown.
\end{proof}

The next result follows from Corollaries \ref{degcor} and \ref{nil2m}.

\begin{proposition}
For any given positive binary form $f\in P_{2,d}$, there exists a unique nonsingular algebra $\mu\in V_{2,2d}$ (up to isomorphism) with Pfaffian form $f_\mu=f$ admitting a nilsoliton inner product.
\end{proposition}

We will see in the next section that this is no longer true for ternary quartics.

We now use Theorem \ref{nilYN} to get explicit continuous families of pairwise non-isomorphic $2$-step nilpotent Lie algebras of type $(2,m)$ which do not admit nilsoliton metrics.  Our examples cover any dimension $\geq 14$ and dimension $12$.  Families of this kind have been exhibited by M. Jablonski in \cite{Jbl} for any dimension $\geq 24$, and in the $3$-step nilpotent case in any dimension $\geq 8$ by T. Payne in \cite{Pyn} (see also \cite{Osc1}).

\begin{proposition}\label{curves2m}
For any $r\geq 1$, each of the following sets $\sca_t$, $t>1$, provides a one-parameter family $\mu_{\sca_t}$ of pairwise non-isomorphic indecomposable $2$-step algebras of type $(2,m)$ which do not admit any nilsoliton metric:
\begin{itemize}
\item $m=8+2r$, $\mathcal{S}_t=\{(\im t,1),(0,2),\overbrace{(1,1),\dots,(1,1)}^r \}$.

\item $m=11+2r$, $\mathcal{S}_t=\{(\im t,1),(0,2),\overbrace{(1,1),\dots,(1,1)}^r,1 \}$.

%\item $m=10$, $\mathcal{S}_t=\{(-1,1),(0,2),(1,1),(t,1)\}$.

\item $m=8+4r$, $\mathcal{S}_t=\{(\im t,2),\overbrace{(\im,1),\dots,(\im,1)}^r \}$.

\item $m=8+4r+5$, $\mathcal{S}_t=\{(\im t,2),\overbrace{(\im,1),\dots,(\im,1)}^{r},(0,1),1\}$.

\item $m=8+4r+2$, $\mathcal{S}_t=\{(\im t,2),\overbrace{(\im,1),\dots,(\im,1)}^r,(0,1)\}$.

\item $m=8+4r+3$, $\mathcal{S}_t=\{(\im t,2),\overbrace{(\im,1),\dots,(\im,1)}^r,1\}$.
\end{itemize}
\end{proposition}

\begin{proof}
In all cases, the non-existence of a nilsoliton follows at once by applying Theorem \ref{nilYN}.

We illustrate how one proves pairwise non-isomorphisms for these cases.  Suppose that two algebras $\mu_{\sca_t}$, $\mu_{\sca_u}$ in the first item are isomorphic.  Thus there exists a M\"obius transformation $T$ as in \eqref{isomuS} such that $T\{0,1,t\im\}=\{0,1,u\im\}$.  If $T0=0$ and $T1=1$, then $b=0$, $a=c+d$ and $a,d\ne 0$, from which it easily follows that $T\im t=\im u$ if and only if $t=u$.  If $T0=1$ and $T1=0$ then one can obtain $t=u$ analogously.

In all the remaining cases, the pairwise non-isomorphism can be easily deduced in much the same way from elementary properties of M\"obius transformations.
\end{proof}

By arguing exactly as above, one can also easily obtain $k$-parameter families of pairwise non-isomorphic indecomposable $2$-step algebras of type $(2,m)$ which do not admit any nilsoliton metric for any $k\geq 1$, even in the nonsingular case.  For example, take $\mu_{\sca_{t_1,\dots,t_k}}$ for
$$
\mathcal{S}_{t_1,\dots,t_k}=\{(\im,2),(\im t_1,1),\dots,(\im t_k,1)\}, \qquad 1\leq t_1\leq\dots\leq t_k,
$$
in which case $m=8+4k$.

\section{Nonsingular algebras of type $(3,8)$}\label{38}

In view of the results obtained in Section \ref{2m}, we have by now a quite clear picture of the nonsingular algebras of type $(2,m)$, including those admitting a nilsoliton metric.  Also, it is well known that there is only one nonsingular algebra of type $(3,4)$ (namely, the quaternionic $H$-type algebra $\ngo=\HH\oplus\Im\HH$).  The next type to study in degree of difficulty is therefore $(3,8)$, which is the aim of this section.

A first problem we have to face in type $(3,8)$ is that the Pfaffian forms are ternary quartics, a very subtle topic from many points of view.  The classification of ternary quartics up to equivalence is an open classical problem in invariant theory, even the ring of invariants $\CC[P_{3,4}]^{\Sl_3(\CC)}$ is not yet completely understood.  Over the real numbers, it may be inferred that an explicit classification for positive ternary quartics is hopeless.  According to a result in \cite{Osc2}, there is at least one nonsingular $\mu\in V_{3,8}^+$ with Pfaffian form a given $f\in P_{3,4}^+$, so that in some sense, a classification of nonsingular algebras of type $(3,8)$ can also be considered out of reach.  One may get an astounding indication of this by looking at the following existence result due to J. Heber (see \cite[Theorem 6.19]{Hbr}): there exists a $12$-parameter continuous family of pairwise non-isomorphic $2$-step algebras of type $(3,8)$ around the $H$-type one, which are all nonsingular in a neighborhood by continuity.  Moreover, they all admit a nilsoliton metric.

\subsection{Existence of nilsolitons}
The following three propositions provide explicit examples of continuous families of nonsingular algebras admitting a nilsoliton inner product with different types of Pfaffian forms.  The existence of a nilsoliton follows in all cases from Lemma \ref{uniform}, as they are all uniform algebras.  The computation of the Pfaffian forms is straightforward, and in Propositions \ref{1-par} and \ref{1-par2}, the fact that the algebras in the families are pairwise non-isomorphic follows from Examples \ref{1-parlap} and \ref{exinvform}, respectively.

\begin{proposition}\label{3-par}
Let $\mu_{t_1,t_2,t_3}\in V_{3,8}$ be the Lie algebra defined by
$$
J_\mu(xZ_1+yZ_2+zZ_3)=
\left[ \begin {smallmatrix} 0&0&x&0&y&0&z&0\\0&0
&0&t_{{1}}x&0&t_{{2}}y&0&t_{{3}}z\\-x&0&0&0&z&0&-y&0
\\0&-t_{{1}}x&0&0&0&z&0&-y\\-y&0&-
z&0&0&0&x&0\\0&-t_{{2}}y&0&-z&0&0&0&x
\\-z&0&y&0&-x&0&0&0\\0&-t_{{3}}z&0
&y&0&-x&0&0\end {smallmatrix} \right].
$$
Then $\mu_{t_1,t_2,t_3}$ with
$$
t_1^2+t_2^2+t_3^2=1,\quad 0<t_3\leq t_2\leq t_1,
$$
is a $2$-parameter family of pairwise non-isomorphic nonsingular algebras admitting nilsolitons with Pfaffian forms given by
$$
f_{\mu_{t_1,t_2,t_3}}(x,y,z)=(x^2+y^2+z^2)(t_1x^2+t_2y^2+t_3z^2).
$$
\end{proposition}

\begin{proof}
It only remains to prove that they are pairwise non-isomorphic, which we shall do via their Pfaffian forms.  We use the fact that a projective equivalence map between $f_{\mu_{t_1,t_2,t_3}}$ and $f_{\mu_{s_1,s_2,s_3}}$ must either preserve or interchange their irreducible factors and that the isotropy subgroup of $x^2+y^2+z^2$ is $\Or(3)$.  It is therefore easy to prove that for $t_i,s_i>0$, $i=1,2,3$, we have that $f_{\mu_{t_1,t_2,t_3}}\simeq f_{\mu_{s_1,s_2,s_3}}$
if and only if there exist $c>0$ and a permutation $\sigma\in\textrm{S}_3$ such that
either $t_i=cs_{\sigma(i)}$ or $t_i=\frac{c}{s_{\sigma(i)}}$ for all $i=1,2,3$, which is impossible for two different triples satisfying the conditions required by the proposition.
\end{proof}

The algebra $\mu_{t_1,t_2,t_3}$ is isomorphic to the $H$-type algebra $\ngo=\HH\oplus\HH\oplus\Im\HH$ if and only if $t_1=t_2=t_3=1/\sqrt{3}$.

\begin{proposition}\label{1-par}
Let $\mu_t\in V_{3,8}$ be defined by
$$
J_\mu(xZ_1+yZ_2+zZ_3)=\left[\begin{smallmatrix}0&0&ax&0&0&y&0&z\\0&0&0&bx&y&0&-z&0\\-ax&0&0&0&0&z&y&0\\0&-bx&0&0&z&0&0&y\\0&-y&0&-z&0&0&0&x\\
-y&0&-z&0&0&0&x&0\\0&z&-y&0&0&-x&0&0\\-z&0&0&-y&-x&0&0&0\end{smallmatrix}\right],
$$
where $a=\tfrac{-t-\sqrt{t^2-4}}{2}$, $b=\tfrac{-t+\sqrt{t^2-4}}{2}$.  Then $\mu_t$ with $t\geq 2$ is a $1$-parameter family of pairwise non-isomorphic nonsingular algebras admitting nilsolitons with Pfaffian forms given by
$$
f_{\mu_t}(x,y,z)={x}^{4}+{y}^{4}+{z}^{4}+t{x}^{2}{y}^{2}.
$$
\end{proposition}

\begin{proposition}\label{1-par2}
Let $\mu_t\in V_{3,8}$ be defined by
$$
J_\mu(xZ_1+yZ_2+zZ_3)= \left[ \begin {smallmatrix} 0&0&-{t}^{2}x&0&y&0&0&\alpha_t\,z
\\0&0&0&-{t}^{2}x&0&y&z&0\\{t}^{2}
x&0&0&0&0&-z&y&0\\0&{t}^{2}x&0&0&-{\frac {z}{\alpha_t}
}&0&0&y\\-y&0&0&{\frac {z}{\alpha_t}}&0&0&x&0
\\0&-y&z&0&0&0&0&x\\0&-z&-y&0&-x&0
&0&0\\-\alpha_t\,z&0&0&-y&0&-x&0&0\end {smallmatrix}
 \right],
$$
where $\alpha_t:=\frac{2t^2}{2t^2+1-\sqrt{4t^2+1}}$.  Then $\mu_t$ with $t>1$ is a $1$-parameter family of pairwise non-isomorphic nonsingular algebras admitting nilsolitons with Pfaffian forms given by
$$
f_{\mu_t}(x,y,z)=\left({x}^{2}+{y}^{2}+{z}^{2} \right) ^{2}+t{x}^{2}{z}^{2}.
$$
\end{proposition}

We note that the families given in the above three propositions are pairwise different, which easily follows from the non-equivalence between their Pfaffian forms.

\subsection{Non-existence of nilsolitons}
The following two propositions provide explicit examples of continuous families of nonsingular algebras of type $(3,8)$ which do not admit a nilsoliton inner product.
The algebras in these families are pairwise non-isomorphic since their Pfaffian forms are pairwise non-equivalent, as we have seen in the proof of Proposition \ref{3-par}.

\begin{proposition}\label{1-parno}
Let $\mu_t\in V_{3,8}$ be defined by
$$
J_{\mu_t}(xZ_1+yZ_2+zZ_3)=\left[ \begin {smallmatrix} 0&0&-x&-y&0&0&0&-z
\\0&0&yt&-x&0&0&-z&0\\x&-yt&0&-x&0
&-z&0&0\\y&x&x&0&-z&0&0&0\\0&0&0&z
&0&0&-x&-y\\0&0&z&0&0&0&y&-x\\0&z&0
&0&x&-y&0&-x\\z&0&0&0&y&x&x&0\end {smallmatrix} \right].
$$
Then $\mu_t$ with $t>1$ is a $1$-parameter family of pairwise non-isomorphic nonsingular algebras with Pfaffian forms given by
$$
f_{\mu_t}(x,y,z)=  \left( {x}^{2}+{y}^{2}+{z}^{2} \right)  \left( {x}^{2}+t{y}^{2}+{z}^{2} \right),
$$
which do not admit a nilsoliton inner-product.
\end{proposition}

\begin{proof}
If we consider $\lambda_s:=\vp_s\cdot\mu_t$, where
$$
\vp_s:=\Diag(e^s,e^s,e^{-s},e^{-s},e^s,e^s,e^{-s},e^{-s},1,1,1)\in\Sl_8\times\Sl_3,
$$
then it is easy to prove that $\lambda_s\to\tilde{\mu}_t$, as $s\to\infty$, where $\tilde{\mu}_t\in V_{3,8}$ is defined by
$$
J_{\tilde{\mu}_t}(xZ_1+yZ_2+xZ_3)=
\left[ \begin{smallmatrix} 0&0&-x&-y&0&0&0&-z
\\0&0&yt&-x&0&0&-z&0\\x&-yt&0&0&0&
-z&0&0\\y&x&0&0&-z&0&0&0\\0&0&0&z&0
&0&-x&-y\\0&0&z&0&0&0&y&-x\\0&z&0&0
&x&-y&0&0\\z&0&0&0&y&x&0&0\end{smallmatrix} \right],
$$
for all $t$.  This implies that $\tilde{\mu}_t\in\overline{\Sl_8\times\Sl_3\cdot\mu_t}$, and since $\dim{\Der(\mu_t)}=26$ and $\dim{\Der(\tilde{\mu}_t)}=29$ (which can be easily computed by using Maple) we have that $\mu_t$ is not isomorphic to $\tilde{\mu}_t$.  It follows that $\Sl_8\times\Sl_3\cdot\mu_t$ is not closed and thus $\mu_t$ does not admit a nilsoliton by Theorem \ref{nilthm}.
\end{proof}

\begin{proposition}\label{1-parno2}
Let $\lambda_t\in V_{3,8}$ be defined by
$$
J_{\lambda_t}(xZ_1+yZ_2+zZ_3)=\left[ \begin {smallmatrix} 0&0&-x&-y&0&0&0&-z
\\\noalign{\medskip}0&0&yt&-x&0&0&-z&0\\\noalign{\medskip}x&-yt&0&-y&0
&-z&-y&-y\\\noalign{\medskip}y&x&y&0&-z&0&-y&-y\\\noalign{\medskip}0&0
&0&z&0&0&-x&-y\\\noalign{\medskip}0&0&z&0&0&0&y&-x\\\noalign{\medskip}0
&z&y&y&x&-y&0&-y\\\noalign{\medskip}z&0&y&y&y&x&y&0\end{smallmatrix}
 \right].
$$
Then $\lambda_t$ with $t>1$ is a $1$-parameter family of pairwise non-isomorphic nonsingular algebras with Pfaffian forms given by
$$
f_{\lambda_t}(x,y,z)=  \left( {x}^{2}+{y}^{2}+{z}^{2} \right)  \left( {x}^{2}+t{y}^{2}+{z}^{2} \right),
$$
which do not admit a nilsoliton inner-product.
\end{proposition}

The proof follows in much the same way as the proof of Proposition \ref{1-parno}, even with the same $\vp_s$ and $\tilde{\mu}_t$, and we now use that $\dim{\Der(\lambda_t)}=27$ for all $t$.  This also shows that the families in the above two propositions are different.

\section{Automorphisms of nonsingular algebras}\label{aut}

We study in this section automorphisms and derivations of $2$-step algebras which are nonsingular.  Our main contribution here is to provide examples other than $H$-type algebras with certain maximality properties concerning the space of derivations.

The Lie algebra of derivations of a $2$-step algebra $\mu\in V_{n,m}$ of type $(n,m)$ decomposes as
\begin{equation}\label{derdec}
\Der(\mu)=\RR \left[\begin{array}{cc} I_m&0\\ 0&2I_n \end{array}\right] \oplus \left[\begin{array}{cc} 0&0\\ \ast&0 \end{array}\right] \oplus \Der_{gr}(\mu),
\end{equation}
where $\Der_{gr}(\mu)$ is the subalgebra of graded derivations given by
$$
\Der_{gr}(\mu):=\left\{ \left[\begin{smallmatrix} B&0\\ 0&A \end{smallmatrix}\right]: B\in\slg_m, \; A\in\glg_n, \;
B^tJ_\mu(Z)+J_\mu(Z)B=J_\mu(A^tZ), \;\forall Z\in\ngo_2\right\}.
$$
Indeed, any linear map taking $\ngo_1$ to $\ngo_2$ and vanishing on $\ngo_2$ is a derivation and any derivation preserving the decomposition $\ngo=\ngo_1\oplus\ngo_2$ can be written as
$$
\left[\begin{smallmatrix} B&0\\ 0&A \end{smallmatrix}\right] =\left[\begin{smallmatrix} \tr{B}/mI_m&0\\ 0&2\tr{B}/mI_n \end{smallmatrix}\right] +\left[\begin{smallmatrix} B-\tr{B}/mI_m&0\\ 0&A-2\tr{B}/mI_n \end{smallmatrix}\right].
$$
This is the Lie algebra of the Lie subgroup of $\Aut(\mu)$ given by
$$
\Aut_{gr}(\mu):=\left\{ \left[\begin{smallmatrix} \psi&0\\ 0&\vp \end{smallmatrix}\right]: \psi\in\Sl_m, \; \vp\in\Gl_n, \;
\psi^tJ_\mu(Z)\psi=J_\mu(\vp^tZ), \;\forall Z\in\ngo_2\right\}.
$$
It follows that $f_\mu(\vp^tZ)=f_\mu(Z)$ for all $Z$, that is, $\vp^t$ is in the automorphism group $\Aut(f_\mu)$ of the Pfaffian form $f_\mu$.  When $\mu$ is nonsingular, one has that $f_\mu>0$, and consequently $\Aut(f_\mu)$ is compact.  The closure of the projection of $\Aut_{gr}(\mu)$ on $\Sl_n=\Sl(\ngo_2)$ is therefore compact, which yields the next result.

\begin{theorem}\label{autcomp}\cite[Theorem 2.1]{KplTrb}
If $\mu$ is nonsingular, then there exists an inner product on the center $\ngo_2$ which is invariant under $\Aut_{gr}(\mu)$.
\end{theorem}

Let us assume from now on that our fixed inner product $\ip$ on $\ngo_2$ is $\Aut_{gr}(\mu)$-invariant for any nonsingular $\mu$ we consider.  By defining the ideal of $\Der_{gr}(\mu)$,
$$
\Der_0(\mu):=\left\{ \left[\begin{smallmatrix} B&0\\ 0&A \end{smallmatrix}\right]\in\Der_{gr}(\mu):A=0\right\},
$$
and the corresponding normal subgroup of $\Aut(\mu)$,
$$
\Aut_0(\mu):=\left\{ \left[\begin{smallmatrix} \psi&0\\ 0&\vp \end{smallmatrix}\right]\in\Aut_{gr}(\mu):\vp=I\right\},
$$
one deduces the following.

\begin{corollary}\label{autcomp2}\cite[Corollary 2.4]{KplTrb}
For any nonsingular $\mu$ of type $(n,m)$,
$$
\dim{\Der_{gr}(\mu)/\Der_0(\mu)}=\dim{\Aut_{gr}(\mu)/\Aut_0(\mu)}\leq \frac{n(n-1)}{2}.
$$
\end{corollary}

\begin{example}\label{Htype}
It is well known that $\dim{\Der_{gr}(\mu)/\Der_0(\mu)}= \frac{n(n-1)}{2}$ for any $H$-type algebra $\mu$ (see \cite{Rhm,Sal}).  Indeed,
$$
\left[\begin{smallmatrix} J_\mu(Z)&0\\ 0&R_Z \end{smallmatrix}\right]\in\Aut(\mu), \qquad \forall Z\in\ngo_2, \quad \|Z\|=1,
$$
where $R_Z(W)=2\la Z,W\ra Z-W$ is the reflection with respect to the hyperplane $\{ Z\}^\perp$ in the center $\ngo_2$.
\end{example}

A natural question arises: which nonsingular algebras of type $(n,m)$ satisfy any of the following equivalent conditions? Do these conditions hold only for $H$-type algebras?

\begin{itemize}
\item $\dim{\Der_{gr}(\mu)/\Der_0(\mu)}= \frac{n(n-1)}{2}$.

\item $\Der_{gr}(\mu)/\Der_0(\mu)\simeq\sog(n)$.

\item For any $A\in\sog(n)$ there exists $B\in\slg_m$ such that $\left[\begin{smallmatrix} B&0\\ 0&A \end{smallmatrix}\right]\in\Der(\mu)$.

\item For any $\vp\in\SO(n)$ there exists $\psi\in\Sl_m$ such that $\left[\begin{smallmatrix} \psi&0\\ 0&\vp \end{smallmatrix}\right]\in\Aut(\mu)$.
\end{itemize}

It is worthwhile noting that these conditions imply that the Pfaffian form is necessarily equal, up to scaling, to $f_\mu=(x_1^2+\dots+x_n^2)^{m/4}$, as this is the only $n$-ary $(m/2)$-ic form which is $\SO(n)$-invariant.  In particular, $\mu$ must be of $\tilde{H}$-type.

We now show that already for type $(2,m)$, there are examples of nonsingular algebras satisfying the above maximality conditions which are not $H$-type.

\begin{proposition}\label{derik}
If $\sca=\{ (\im,k_1),\dots,(\im,k_r)\}$, then the corresponding nonsingular $2$-step algebra $\mu_{\sca}$ of type $(2,m)$ (see Section \ref{2m}) satisfies that  $\Der_{gr}(\mu_{\sca})/\Der_0(\mu_{\sca})\simeq\sog(2)$.
\end{proposition}

\begin{remark}
This result has been independently proved in \cite[Theorem 3.5]{KplTrb}, where it is in addition obtained that $\Der_{gr}(\mu_{\sca})/\Der_0(\mu_{\sca})=0$ for any other nonsingular algebra of type $(2,m)$.
\end{remark}

\begin{proof}
Suppose first that $\sca=\{ (\im,k)\}$.  It is an involved but straightforward computation to see that
$$
\left[\begin{smallmatrix} M_1&&&\\ &M_2&&\\ &&0&1\\ &&-1&0 \end{smallmatrix}\right]\in\Der(\mu_{\sca}),
$$
where $M_1$ and $M_2$ are the $(2k\times 2k)$-matrices defined by
{\small
$$
M_1=\left[
  \begin{array}{cc|cc|cc|cc|cc|cc}
    0& 0&-1& 0&  &  &  &  &  &  &  &  \\
    0& 0& 0&-1&  &  &  &  &  &  &  &  \\ \hline
     &  & 0& 2& 0&  &  &  &  &  &  &  \\
     &  &-2& 0&  & 0&  &  &  &  &  &  \\ \hline
     &  &  &  & 0& 4& 1& 0 &  &  &  &  \\
     &  &  &  &-4& 0& 0& 1 &  &  &  &  \\ \hline
     &  &  &  &  &  & 0& 6& \ddots &  &  &  \\
     &  &  &  &  &  &-6& 0& & \ddots  &  &  \\ \hline
     &  &  &  &  &  &  &  & \ddots& & k-3 & 0 \\
     &  &  &  &  &  &  &  & & \ddots& 0 & k-3 \\ \hline
     &  &  &  &  &  &  &  &  &  & 0 & 2(k-1) \\
     &  &  &  &  &  &  &  &  &  & -2(k-1) & 0 \\
  \end{array}
\right] \
$$
and
$$
M_2=\left[
  \begin{array}{cc|cc|cc|cc|cc|cc|cc}
    0 & -(2k-1) & -(k-2) & 0 &  &  &  &  &  &  &  &  &  \\
    (2k-1) & 0 & 0 & -(k-2) &  &  &  &  &  &  &  &  &  \\ \hline
         &  & 0 & -(2k-3) & \ddots &  & &  &  &  &  &  &  \\
         &  & (2k-3) & 0 & & \ddots & &  &  &  &  &  &  \\ \hline
         &  &  &  & \ddots & & -3& 0&  &  &  &  &  \\
         &  &  &  &  & \ddots & 0&-3&  &  &  &  &  \\ \hline
         &  &  &  & & & 0 & -7  & -2 & 0 &  &  &  \\
         &  &  &  & & & 7 & 0  & 0 & -2 &  & &  \\ \hline
     &  &  &  &  &  &  & & 0 & -5 & -1 & 0 &  \\
     &  &  &  &  &  &  & & 5 & 0 & 0 & -1 &  \\ \hline
     &  &  &  &  &  &  & &   &   & 0 &-3 & 0 & 0 \\
     &  &  &  &  &  &  &  &  &   & 3 & 0 & 0 & 0 \\ \hline
     &  &  &  &  &  &  &  &  &   &   &   & 0 &-1 \\
     &  &  &  &  &  &  &  &  &   &   &   & 1 & 0
  \end{array}
\right].
$$ }
One can use the formula $J_\mu(Z_i)=\left[\begin{smallmatrix} 0&A_i\\ -A_i^t&0 \end{smallmatrix}\right]$ given in (\ref{Jcomplex}) and check that
$$
M_1^tA_1+A_1M_2=A_2, \qquad M_1^tA_2+A_2M_2=-A_1.
$$
Finally, if $\sca=\{ (\im,k_1),\dots,(\im,k_r)\}$, then we can define the required derivation made of blocks as above, concluding the proof.
\end{proof}

Note that an algebra $\mu_{\sca}$ as in Proposition \ref{derik} is of $H$-type if and only if $k_1=\dots =k_r=1$, which is also equivalent to the existence of a nilsoliton metric (see Corollary \ref{nil2m}).

Let us now consider skew-symmetric derivations and orthogonal automorphisms.  For each algebra $\mu$ of type $(n,m)$ we have that
$$
\kg(\mu):=\Der(\mu)\cap\sog(\ngo)= \left\{ \left[\begin{smallmatrix} B&0\\ 0&A \end{smallmatrix}\right]: B\in\sog(m), \; A\in\sog(n), \;
[B,J_\mu(Z)]=J_\mu(AZ), \;\forall Z\in\ngo_2\right\},
$$
and if
$$
\kg_0(\mu):=\left\{ \left[\begin{smallmatrix} B&0\\ 0&A \end{smallmatrix}\right]\in\kg(\mu):A=0\right\},
$$
then it is also natural to ask ourselves which nonsingular algebras of type $(n,m)$ satisfy
\begin{equation}\label{max}
\kg(\mu)/\kg_0(\mu)\simeq\sog(n).
\end{equation}
This stronger property is requiring any $A\in\sog(n)$ being extendable to a skew-symmetric derivation of $\mu$.  $H$-type algebras also satisfy this (see Example \ref{Htype}), but it cannot be deduced from the proof of Proposition \ref{derik} that the algebras $\mu_{\sca}$'s considered there do.

The corresponding groups of automorphisms are given by
$$
K(\mu):=\Aut(\mu)\cap\Or(\ngo)=\left\{ \left[\begin{smallmatrix} \psi&0\\ 0&\vp \end{smallmatrix}\right]: \psi\in\SO(m), \; \vp\in\SO(n), \;
\psi J_\mu(Z)\psi^{-1}=J_\mu(\vp Z), \;\forall Z\in\ngo_2\right\}.
$$
and
$$
K_0(\mu):=\left\{ \left[\begin{smallmatrix} \psi&0\\ 0&\vp \end{smallmatrix}\right]\in K(\mu):\vp=I\right\}.
$$

\subsection{$2$-step algebras of $Rep$-type}
We now consider a class of $2$-step algebras attached to finite-dimensional real representations of compact Lie groups.  Our aim is to show that condition (\ref{max}) can also hold for nonsingular algebras which are not of $H$-type.

\begin{definition}\label{reptype}
A $2$-step algebra $\mu\in V_{n,m}$ is said to be of {\it Rep-type} if the following holds:
\begin{itemize}
\item $\ngo_2$ is a compact Lie algebra (say with Lie bracket $\lb$).

\item $J_\mu:\ngo_2\longrightarrow\sog(\ngo_1)$ is a representation of $\ngo_2$, i.e.
$$
J_\mu([Z,W])=[J_\mu(Z),J_\mu(W)], \qquad \forall Z,W\in\ngo_2.
$$
\end{itemize}
\end{definition}

We refer to \cite{EbrHbr,manus,Ebr3} for more detailed expositions on $Rep$-type algebras.  Some of their nice and interesting properties are now listed, all of which can be easily checked.

\begin{itemize}
\item $\mu$ is of type $(n,m)$ if and only if the representation $J_\mu$ is faithful.

\item We can assume from now on that the fixed inner product we have on $\ngo_2$ is $\ad{\ngo_2}$-invariant, i.e. any adjoint map $\ad{Z}$ of the Lie algebra $\ngo_2$ is skew-symmetric.

\item Each $Z\in\ngo_2$ defines a skew-symmetric derivation of $\mu$ by
$$
\left[\begin{smallmatrix} J_\mu(Z)&0\\ 0&\ad{Z} \end{smallmatrix}\right]\in\Der(\mu).
$$
(see \cite[Theorem 3.12]{manus}).

\item $\mu$ is nonsingular if and only if $\ngo_2=\sug(2)$ ($=\sog(3)$) and any irreducible subrepresentation of $(\ngo_1,J_\mu)$ is even dimensional (or equivalently, a copy of the $2(k+1)$-dimensional $\sug(2)$-representation $P_{2,k}(\CC)$, $k$ odd, of binary $k$-ic complex forms viewed as real, see e.g. \cite{BrctDc}).  This can be deduced as follows: in order to have that $J_\mu(Z)$ is invertible for any nonzero $Z$ it is necessary and sufficient that $\rank(\ngo_2)=1$ and the weights of the representation $(\ngo_1,J_\mu)$ being all nonzero.
\end{itemize}

We therefore obtain the following result by combining the above observations.

\begin{proposition}
If $\mu\in V_{3,m}$ is a nonsingular algebra of $Rep$-type, then
$$
\kg(\mu)/\kg_0(\mu)\simeq\sog(3).
$$
\end{proposition}

Therefore, any nonsingular $\mu$ of $Rep$-type of type $(3,m)$, which is not of $H$-type, provides a counterexample to \cite[Theorem 2.5]{KplTrb}.

It is easy to see that a $Rep$-type $\mu$ is of $H$-type if and only if $\ngo_2=\sug(2)$ and $(\ngo_1,J_\mu)$ is a sum of copies of the standard representation $\CC^2$ of $\sug(2)$. The simplest nonsingular $\mu$ of $Rep$-type not of $H$-type is $\ngo_2=\sug(2)$ and the representation $(\ngo_1,J_\mu)=P_{2,3}(\CC)=\CC^4$ viewed as real.  This is precisely the example of type $(3,8)$ considered in Example \ref{su2}, and an elementary way to see that it is not of $H$-type is by noticing that $J_\mu(\ngo_2)$ acts irreducibly on $\ngo_1$.

Any $Rep$-type $2$-step algebra admits a nilsoliton metric (see \cite{EbrHbr,cruzchica}).  Under the assumptions that $\ngo_2$ is simple and $J_\mu$ is irreducible, we have that $(\ngo,\mu,\ip)$ is a nilsoliton itself.  Indeed, since $(Z,W):=\tr{J_\mu(Z)J_\mu(W)}$ is $\ad{\ngo_2}$-invariant and the Casimir operator $\sum J_\mu(Z_i)^2$ intertwines the representation, we obtain that condition (\ref{nilso}) holds by Schur's Lemma.

\end{document}